\documentclass[a4paper,oneside,onecolumn]{article}
\usepackage{cancel}
 
\usepackage{soul}
\usepackage[left=2.0cm,top=2.5cm,right=2.0cm,bottom=2.5cm,bindingoffset=0.0cm]{geometry}

\usepackage{float} 
\usepackage{multirow}
\usepackage[T1]{fontenc} 

\usepackage{pgfplots}
\pgfplotsset{compat=1.18}
\usepgfplotslibrary{groupplots}

\usepackage[utf8]{inputenc} 
\newcommand{\vertiii}[1]{{\left\vert\kern-0.25ex\left\vert\kern-0.25ex\left\vert #1 
    \right\vert\kern-0.25ex\right\vert\kern-0.25ex\right\vert}}
\usepackage{dsfont}
\usepackage{graphics,graphicx,epstopdf} 

\usepackage{enumitem} 
\usepackage{mathrsfs}

\usepackage{subfig} 

\usepackage{amsmath,amssymb,amsthm,amsfonts} 
\usepackage{mathtools} 
\usepackage{booktabs}             

\allowdisplaybreaks

\usepackage{setspace} 
\usepackage{lineno, hyperref} 

\usepackage{color, xcolor}

\usepackage{authblk} 

\theoremstyle{plain} 
\newtheorem{theorem}{Theorem}[section]
\newtheorem{lemma}{Lemma}[section]

\newtheorem{assumption}{Assumption}[section]

\newtheorem{example}{Example}[section]
\theoremstyle{remark} 
\newtheorem{remark}{Remark}[section]

\definecolor{c1}{rgb}{0,0,1} 
\definecolor{c2}{rgb}{0,0.6,1} 
\definecolor{c3}{rgb}{0.5,0,0.5} 

\hypersetup{
	pdfauthor={Lok Pati},
	pdfsubject={},
	pdftitle={},
	pdfkeywords={},
	breaklinks = true, 
	linktocpage=true, 
	colorlinks=true,       
	linkcolor={c1},          
	citecolor={c2},        
	filecolor=black,      
	urlcolor={c3},       
}
\numberwithin{equation}{section}

\allowdisplaybreaks
\makeatletter
\def\namedlabel#1#2{\begingroup
    #2%
    \def\@currentlabel{#2}%
    \phantomsection\label{#1}\endgroup
}
\makeatother
\begin{document}


\title{\textbf{Sharp Error Estimates for a Fully Discrete Finite Element Method for Semilinear SPDEs with Multiplicative Noise and Nonsmooth Initial Data}}

\author[1]{\textsc{Jitendra Nath Naik}\thanks{\href{mailto:jitendra20232201@iitgoa.ac.in}{jitendra20232201@iitgoa.ac.in}}}
\author[1]{\textsc{Lok Pati Tripathi}\thanks{\href{mailto:lokpati@iitgoa.ac.in}{lokpati@iitgoa.ac.in}}}
\affil[1]{School of Mathematics and Computer Science, Indian Institute of Technology Goa, Goa 403401, India.}

\date{} 

\maketitle 


\begin{abstract}
This article establishes sharp strong error estimates for the fully discrete approximation of semilinear parabolic stochastic partial differential equations (SPDEs) driven by multiplicative noise and subject to nonsmooth initial data. The spatial discretization is based on a standard finite element method, coupled with the linearly implicit Euler scheme in time. By mapping the diffusion operator into negative fractional spaces, our framework accommodates both trace-class and space-time white noise. For nonsmooth initial data, by decoupling the noise regularity parameter $\beta \in (0,2)$ from the initial data regularity parameter $\mu \in (0,2]$, we derive sharp regularity estimates that isolate the exact loss of initial regularity into an integrable temporal singularity. Furthermore, we establish sharp strong convergence rates of $O(h^{\beta-\varepsilon} + k^{\frac{1}{2}\min\{\beta-\varepsilon, 1\}})$ for $\varepsilon>0$ away from $t = 0$. Finally, we consider physically relevant stochastic models, such as the modified Langmuir fractional surface coverage model and the parabolic Anderson model, in our numerical experiments to confirm the theoretical convergence rates.
\end{abstract}


\vspace{1em}


\vspace{2em} 

\section{Introduction}\label{sec:introduction}

Let $(\Omega, \mathcal{F}, \mathbb{P})$ be a probability space equipped with a normal filtration $\{\mathcal{F}_t\}_{t \in [0,T]}$, and let $W$ be a $Q$-Wiener process (possibly cylindrical) on a separable Hilbert space $(U, (\cdot, \cdot)_U)$ with respect to $\{\mathcal{F}_t\}_{t \in [0,T]}$, where the covariance operator $Q \colon U \to U$ is linear, bounded, self-adjoint, and positive semidefinite. Let $\mathcal{O}\subset\mathbb{R}^{d}$ ($d \in \{1, 2, 3\}$) be either a convex bounded domain or a bounded domain with a $C^{1,1}$ boundary $\partial\mathcal{O}$. We consider the following semilinear stochastic convection-diffusion equation driven by multiplicative noise:
\begin{equation}\label{eq:SPDE}
\left\{
\begin{aligned}
    du(t,x) &= \big[ \nabla \cdot (\boldsymbol{a}(x)\nabla u(t,x)) - \boldsymbol{b}(x) \cdot \nabla u(t,x) - c(x)u(t,x) + f(t, x, u(t,x)) \big] dt \\
            &\quad \quad+ g(t, x, u(t,x)) \, dW(t), && (t,x) \in (0,T] \times \mathcal{O}, \\[1ex]
    u(t,x) &= 0, && (t,x) \in (0,T] \times \partial\mathcal{O}, \\[1ex]
    u(0,x) &= u_0(x), && x \in \mathcal{O}.
\end{aligned}
\right.
\end{equation}
where the coefficients $\boldsymbol{a} \colon \mathcal{O} \to \mathbb{R}^{d \times d}$, $\boldsymbol{b} \colon \mathcal{O} \to \mathbb{R}^{d}$, and $c \colon \mathcal{O} \to \mathbb{R}$ are sufficiently smooth and bounded. The nonlinear functions $f \colon [0,T] \times \mathcal{O} \times \mathbb{R} \to \mathbb{R}$ and $g \colon [0,T] \times \mathcal{O} \times \mathbb{R} \to \mathbb{R}$ are given measurable functions. We assume $\boldsymbol{a}(x)$ is symmetric and uniformly positive definite; that is, there exists a constant $a_0 > 0$ such that
$$y^T \boldsymbol{a}(x) y \ge a_0 |y|^2 \quad \text{for all } x \in \mathcal{O} \text{ and } y \in \mathbb{R}^d.$$

Following the standard approach to infinite-dimensional stochastic integration (see, e.g., Da Prato and Zabczyk~\cite{MR3236753}, Pr\'{e}v\^{o}t and R\"{o}ckner~\cite{MR2329435}), we let $U = L^2(\mathcal{O})$ and introduce the separable Hilbert space $U_0 \coloneqq Q^{1/2}(U)$, which is equipped with the inner product
\[
    (u_0, v_0)_{U_0} \coloneqq (Q^{-1/2}u_0, Q^{-1/2}v_0)_U, \quad \text{for } u_0, v_0 \in U_0,
\]
where $Q^{-1/2}$ denotes the pseudoinverse of $Q^{1/2}$ in the case where $Q$ is not injective. Furthermore, the space of all Hilbert--Schmidt operators $\Phi \colon U_0 \to H$ is denoted by $\mathrm{HS}(U_0, H)$. This forms a separable Hilbert space equipped with the norm
\[
    \|\Phi\|_{\mathrm{HS}(U_0, H)} \coloneqq \left( \sum_{m=1}^\infty \|\Phi \psi_m\|_H^2 \right)^{1/2},
\]
where $\{\psi_m\}_{m \ge 1}$ is an arbitrary orthonormal basis of $U_0$.

To study \eqref{eq:SPDE} analytically and prepare for the finite element discretization, we recast it as an abstract stochastic evolution equation in the separable Hilbert space $H = L^2(\mathcal{O})$. We define the linear operator $A \colon \mathcal{D}(A) \subset H \to H$ by $A v = -\nabla \cdot (\boldsymbol{a}(x) \nabla v)$, subject to homogeneous Dirichlet boundary conditions. The operator $-A$ generates an analytic $\mathcal{C}_0$-semigroup of contractions on $H$, denoted by $S(t) = e^{-tA}$ for $t \ge 0$. Applying the spectral theorem to $A^{-1}$ yields an orthonormal basis of eigenfunctions for $A$, allowing us to define the fractional powers $A^\gamma$ for $\gamma \in \mathbb{R}$ (see Kruse~\cite[Appendix~B.2]{MR3154916}). The corresponding fractional spaces are defined as $\dot{H}^\gamma \coloneqq \mathcal{D}(A^{\gamma/2})$, equipped with the norm $\|u\|_{\dot{H}^\gamma} \coloneqq \|A^{\gamma/2}u\|_H$. Let $\beta \in (0,2)$ be a given noise regularity parameter. We define the Nemytskii operators $F \colon [0,T] \times H \to \dot{H}^{-1}$ and $G \colon [0,T] \times H \to \mathrm{HS}(U_0, \dot{H}^{\beta-1})$ by 
\begin{align}
    (F(t, \phi))(x) &\coloneqq f(t, x, \phi(x)) - \boldsymbol{b}(x) \cdot \nabla \phi(x) - c(x) \phi(x), \label{eq:Nemytskii_F} \\[1ex]
    (G(t, \phi)\psi)(x) &\coloneqq g(t, x, \phi(x)) \psi(x), \label{eq:Nemytskii_G}
\end{align}
for all $t \in [0,T]$, $x \in \mathcal{O}$, $\phi \in H$, and $\psi \in U_0$. This permits the SPDE \eqref{eq:SPDE} to be recast as the following abstract stochastic evolution equation for the $H$-valued stochastic process $X(t) \coloneqq u(t, \cdot)$:
\begin{equation}\label{eq:SEE}
\begin{aligned}
    dX(t) + A X(t) \, dt &= F(t, X(t)) \, dt + G(t, X(t)) \, dW(t), \quad t \in (0, T], \\
    X(0) &= X_0.
\end{aligned}
\end{equation}
where $X_0$ is $\mathcal{F}_0$-measurable and $F$ and $G$ are the nonlinear drift and diffusion operators defined above. The precise assumptions on the coefficients are detailed in Section~\ref{sec:assumptions}.

The existence, uniqueness, and regularity of an $H$-valued mild solution to \eqref{eq:SEE} are well established in the literature (see, e.g., the standard monographs \cite{MR2329435, MR2856611, MR3154916, MR3236753, MR3308418, MR3410409}). The solution process, denoted by $X \colon [0,T] \times \Omega \to H$, is formulated through the following stochastic integral equation:
\begin{equation}\label{eq:mild_solution_SEE}
    X(t) = S(t)X_{0} + \int_{0}^{t} S(t-\sigma)F(\sigma,X(\sigma))\,d\sigma + \int_{0}^{t} S(t-\sigma)G(\sigma,X(\sigma))\,dW(\sigma).
\end{equation}
While existence and uniqueness is well established, deriving convergence rates for numerical approximations demands precise spatial and temporal regularity results (for a detailed overview of SPDE regularity, see, e.g., \cite{MR2852200, MR2968672, MR3154916}). A critical nuance in SPDEs with multiplicative noise is the maximal spatial regularity barrier. Although SPDEs driven by additive noise can achieve $\dot{H}^2$ spatial regularity, multiplicative noise fails to reach this limit under standard Lipschitz and linear growth assumptions (see \cite{MR2852200, MR2968672}). This barrier arises because the mild solution is at most Hölder continuous in time with an exponent less than $1/2$. This lack of temporal regularity is insufficient to compensate for the non-integrable $(t-\sigma)^{-1}$ singularity generated by the analytic semigroup inside the stochastic convolution (see Theorem~\ref{thm:singular_regularity}). However, $\dot{H}^2$ spatial regularity for multiplicative noise can be achieved under stronger assumptions on the diffusion coefficient \(G\). Section~\ref{sec:assumptions} provides the well-posedness and base regularity of the mild solution (Theorem~\ref{thm:base_regularity}) based on our earlier work \cite{NaikTripathi2026}. Building on these results, Section~\ref{sec:regularity_results} establishes the sharp spatial and temporal regularity bounds (Theorem~\ref{thm:singular_regularity}) essential for the numerical error analysis with nonsmooth initial data.

Numerical approximations of the SPDE \eqref{eq:SEE} have gained significant attention in the last two decades, as exact analytical solutions are rarely available. This has led to an extensive body of literature on spatial and temporal discretizations \cite{MR1637047, MR1683281, MR1873517, MR1953619, MR2182132, MR4050540}. Strong error estimates for finite element approximations have been widely investigated \cite{MR2211047, MR2182132, MR3047942, MR3168284, MR3274889, MR3872387, MR4876549, MR3761284}. For multiplicative noise, a fundamental limitation is that the temporal convergence rate of the Euler--Maruyama scheme is capped at $O(k^{1/2})$. In the context of Galerkin finite element methods, Yan~\cite{MR2182132} derived strong convergence rates of $O(h^\beta + k^{\gamma/2})$ (for $0 \le \gamma < \beta \le 1$). Subsequently, Kruse~\cite{MR3168284} established optimal strong rates of $O(h^{1+r} + k^{1/2})$ for $r \in [0,1)$. More recently, Tambue and Mukam~\cite{MR3872387} extended this framework to non-self-adjoint operators. Crucially, the strong convergence results established in \cite{MR2182132, MR3168284, MR3872387} are all restricted to trace-class noise.

Mathematical models of physical phenomena, such as phase transitions or localized thermal diffusion, frequently require the use of nonsmooth initial conditions, including indicator functions. The presence of these initial singularities typically leads to a reduction in the provable order of convergence for numerical schemes. To circumvent this order reduction in deterministic settings, finite element error analyses are conducted by relying on the smoothing properties of the analytic semigroup (see, e.g., \cite{MR906175, MR906176, MR970700, MR2249024}). Recovering optimal rates of finite element methods for SPDEs with nonsmooth initial data is nontrivial. To the best of our knowledge, such analysis has only been achieved for SPDEs driven by additive noise \cite{naik2026optimalerrorestimatesfinite}. The combined challenge of nonsmooth initial data and multiplicative space-time white noise remains unaddressed. We address this gap in the present article by establishing sharp spatial and temporal convergence rates for semilinear SPDEs subject to nonsmooth initial data. Our analytical approach relies on decoupling the noise regularity parameter $\beta$ from the initial data regularity parameter $\mu$, unlike existing works that impose the constraint $\beta = \mu$ \cite{MR2182132, MR3168284, MR3872387}. By doing so, we capture the exact loss of initial regularity within an integrable temporal weight that blows up as $t \to 0$, thereby obtaining sharp convergence rates for $t > 0$.

To illustrate the physical applicability of Assumptions~\ref{ass:diffusion} and \ref{ass:drift}, we consider the modified Langmuir nonlinearity and the parabolic Anderson model. Originally developed in 1918, the deterministic Langmuir model describes the adsorption of molecules from a fluid to a solid surface using the fractional surface coverage function $f(u) = \frac{u}{1+u}$ (see Langmuir~\cite{langmuir1918adsorption}). However, in a stochastic model, random noise can push the concentration into negative values, which is physically impossible; thus, we consider the modified Langmuir fractional surface coverage function $f(u) = \frac{u}{1+|u|}$ to preserve symmetry and prevent unbounded growth (see Mukam and Tambue \cite{10.1093/imanum/draf152}). This modified drift satisfies our Assumption~\ref{ass:drift}. A second, independent difficulty in the multiplicative noise setting concerns the space in which the diffusion coefficient $G$ is posed. Standard analyses require the diffusion coefficient $G$ to satisfy Lipschitz and linear growth conditions as a mapping into $\mathrm{HS}(U_0, H)$ (see, e.g., \cite{MR2182132, MR3168284, MR3872387, MR3761284}). This restriction forces the covariance operator $Q$ to be of trace class. For space-time white noise ($Q = I$), a Nemytskii-type multiplication operator is never Hilbert--Schmidt into $H$. Consequently, the standard hypothesis cannot accommodate space-time white noise. We overcome this barrier in Assumption~\ref{ass:diffusion} by letting $G$ take values in $\mathrm{HS}(U_0, \dot{H}^{\beta-1})$, where $\beta \in (0,2)$ characterizes the spatial regularity of the noise. Evaluating the operator in a space of negative order introduces decaying spectral weights that keep the Hilbert--Schmidt norms finite even for $Q=I$. Under the formulation \eqref{eq:Nemytskii_G}, this framework allows us to handle standard test models from the SPDE literature. Specifically, we can accommodate bounded diffusion coefficients such as $g(u) = \frac{u}{1+u^2}$ and $g(u) = \frac{1-u^2}{1+u^2}$, as well as the linear diffusion $g(u) = u$ of the parabolic Anderson model (see, e.g., Section~4 in Jentzen and R\"ockner \cite{MR3320928}). The parabolic Anderson model characterizes the transport of a substance through a random medium or the evolution of a population within a random environment (see \cite{MR1185878}). Crucially, our framework allows us to simulate these physically relevant models under both trace-class and space-time white noise (see Section~\ref{sec:numerical_results}).

The primary contributions of this article are outlined below. Under the nonsmooth initial condition \(X_0\in L^p(\Omega;\dot H^\mu)\), with \(\mu\in(0,2]\),

\begin{itemize}[leftmargin=4em]
\item sharp spatial and temporal regularity estimates are established for the mild solution of \eqref{eq:SEE}; see Theorem~\ref{thm:singular_regularity}.

\item for the fully discrete approximation obtained by combining the standard finite element method with the linearly implicit Euler scheme, the strong convergence estimate
\[
O\!\big(h^{\beta-\varepsilon}
+k^{\frac12\min\{\beta-\varepsilon,1\}}\big)
\]
is established for every fixed \(t>0\) and \(\varepsilon\in(0,\beta)\); see Theorem~\ref{thm:strong_convergence_of_fully_discrete_scheme}. In contrast to \cite{MR2182132, MR3168284, MR3872387}, where the regularity of the initial data is tied to that of the noise, the present framework allows the initial-data regularity \(\mu\in(0,2]\) and the noise regularity \(\beta\in(0,2)\) to be prescribed independently. In addition, the analysis is carried out under weaker assumptions on the drift and diffusion coefficients \(F\) and \(G\) than those imposed in \cite{MR2182132, MR3168284, MR3308418, MR3872387}. 

\item numerical experiments are presented to verify the predicted spatial and temporal convergence rates; see Section~\ref{sec:numerical_results}.
\end{itemize}

The remainder of the paper is organized as follows. Section~\ref{sec:assumptions} provides the mathematical setting and assumptions. Section~\ref{sec:fullydiscrete_approximation} formulates the fully discrete schemes and states our main result (Theorem~\ref{thm:strong_convergence_of_fully_discrete_scheme}). Section~\ref{sec:regularity_results} derives the required regularity estimates and Section~\ref{sec:proof_of_fully_discrete_result} the proof of the main theorem, followed by numerical experiments in Section~\ref{sec:numerical_results}.

\section{Setting and assumptions}\label{sec:assumptions}
Throughout this article, we retain the functional setting introduced in Section~\ref{sec:introduction}. We denote by $\mathcal{L}(H)$ the space of bounded linear operators on $H$, equipped with the standard operator norm $\|\cdot\|_{\mathcal{L}(H)}$. The letter $C$ represents a generic positive constant whose exact value may change from line to line, with any essential parameter dependencies indicated explicitly.

To guarantee the existence of a unique solution to \eqref{eq:SEE} and to establish the strong error estimates, we impose the following assumptions.

\begin{assumption}\label{ass:initial_data}
Let $p \ge 2$ and $\mu \in (0,2]$. Let the initial data $X_0 \colon \Omega \to H$ be measurable from the measurable space $(\Omega, \mathcal{F}_0)$ to $(H, \mathcal{B}(H))$, and satisfy
\[
    X_0 \in L^p\bigl(\Omega; \dot{H}^{\mu}\bigr).
\]
\end{assumption}

\begin{assumption}\label{ass:diffusion}
Let $\beta \in (0,2)$ be a parameter characterizing the spatial regularity of the noise. We assume the mapping $G \colon [0,T] \times H \to \mathrm{HS}(U_0, \dot{H}^{\beta-1})$ satisfies the following conditions: there exists a constant $C > 0$ such that for all $t, s \in [0,T]$,
\begin{align}
    \|G(t,u) - G(s,v)\|_{\mathrm{HS}(U_0, \dot{H}^{\min(0,\beta-1)})} &\leq C \bigl( |t-s|^{\beta/2} + \|u-v\|_H \bigr), &&\text{for all } u, v \in H, \\
    \|G(t,u)\|_{\mathrm{HS}(U_0, \dot{H}^{\beta-1})} &\leq C \bigl(1 + \|u\|_H\bigr), &&\text{if } 0 < \beta \leq 1 \text{ and } u \in H, \\
    \|G(t,u)\|_{\mathrm{HS}(U_0, \dot{H}^{\beta-1})} &\leq C \bigl(1 + \|u\|_{\dot{H}^{\beta-1}}\bigr), &&\text{if } 1 < \beta < 2 \text{ and } u \in \dot{H}^{\beta-1}.
\end{align}
\end{assumption}

\begin{assumption}\label{ass:drift}
For the parameter $\beta \in (0,2)$ given in Assumption~\ref{ass:diffusion}, the mapping $F \colon [0,T] \times H \to \dot{H}^{-1}$ satisfies the following conditions: there exists a constant $C>0$ such that for all $t, s \in [0,T]$:
\begin{align}
    \|F(t,u)-F(s,v)\|_{\dot{H}^{-1}} &\le C\,\left( |t-s|^{\beta/2} + \|u-v\|_H\right), &&\text{for all } u, v \in H, \\
    \|F(t,v)\|_{\dot{H}^{-1}} &\le C\bigl(1 + \|v\|_H\bigr), &&\text{for all } v \in H.
\end{align}
Additionally, if $1 < \beta < 2$, $F$ satisfies the following:
\begin{align}
    \|F(t,u)-F(s,v)\|_{\dot{H}^{\beta-2}} &\le C\,\left( |t-s|^{\beta/2} + \|u-v\|_{\dot{H}^{\beta-1}}\right), &&\text{for all } u, v \in \dot{H}^{\beta-1}, \\
    \|F(t,v)\|_{\dot{H}^{\beta-2}} &\le C\bigl(1 + \|v\|_{\dot{H}^{\beta-1}}\bigr), &&\text{for all } v \in \dot{H}^{\beta-1}.
\end{align}
\end{assumption}

\begin{remark}\label{rem:weaker_assumptions}
\mbox{}
\begin{enumerate}[label={\upshape(\roman*)}]
    \item In Assumption~\ref{ass:drift}, we relax the condition on the drift term $F$ by allowing it to map into the negative-order space $\dot{H}^{-1}$. This broadens standard assumptions found in the literature, which typically restrict $F$ to $H$ or $\dot{H}^{-1+r}$ with $r \ge 0$ (see, e.g., \cite{MR3168284, MR3872387}), and allows us to include non-self-adjoint operators.

    \item In Assumption~\ref{ass:diffusion}, we permit the diffusion coefficient $G$ to take values in $\mathrm{HS}(U_0, \dot{H}^{\beta-1})$ rather than the standard space $\mathrm{HS}(U_0, H)$ (see, e.g., \cite{MR3168284, MR3308418, MR3872387}). This enables us to accommodate multiplicative space-time white noise.
\end{enumerate}
\end{remark}

The following well-posedness and regularity results follow directly from the abstract framework developed in \cite{NaikTripathi2026}.

\begin{theorem}[Well-posedness and regularity]
\label{thm:base_regularity}
Under Assumptions~\ref{ass:initial_data}--\ref{ass:drift}, problem~\eqref{eq:SEE} admits a unique mild solution $X$. Furthermore, for $p \in [2, \infty)$, the following properties hold:
\begin{enumerate}[label={\upshape(\roman*)}]
    \item There exists a constant $C_1 > 0$ depending on $T$, such that
    \begin{equation}\label{eq:base_spatial_regularity}
        \sup_{t \in [0,T]} \|X(t)\|_{L^p(\Omega; H)} \le C_1 \bigl( 1 + \|X_0\|_{L^p(\Omega; H)} \bigr).
    \end{equation}
    
    \item For any $0 \le s < t \le T$, there exists a constant $C_2 > 0$ depending on $T$ and $X_0$ such that
   \begin{equation}\label{eq:base_temporal_regularity}
        \|X(t) - X(s)\|_{L^p(\Omega; H)} \le C_2 \, (t-s)^{\delta} \quad \text{for } 0 < \delta < \frac{1}{2}\min(1, \mu, \beta).
    \end{equation}
\end{enumerate}
\end{theorem}

\begin{proof}
The proof follows directly from \cite[Theorem~2.1]{NaikTripathi2026}, as Assumptions~\ref{ass:initial_data}--\ref{ass:drift} satisfy the required hypotheses of the theorem. (i) By setting the abstract base spaces in \cite{NaikTripathi2026} to $H$ (corresponding to parameters $\mu=0$ and $\nu=0$), the solution space $\mathbb{V}_p$ coincides with the space of predictable processes in $\mathcal{C}([0,T]; L^p(\Omega; H))$. The uniform spatial bound \eqref{eq:base_spatial_regularity} is thus an immediate consequence of the corresponding $\mathbb{V}_p$-norm estimate.

(ii) To establish the temporal regularity \eqref{eq:base_temporal_regularity}, we utilize the initial data condition $X_0 \in L^p(\Omega; \dot{H}^\mu)$. By connecting this to the global H\"older continuity framework in \cite[Theorem~2.1(iii)]{NaikTripathi2026} (specifically setting the initial data regularity parameter $\varepsilon= \min(1, \mu)$, the drift codomain index to $\alpha=1$, and the diffusion codomain index to $(1-\beta)^+$), the theorem's general exponent formula evaluates to $0 < \delta < \min\bigl\{\frac{1}{2}\min(1, \mu), \frac{1}{2}, \frac{1}{2}\min(1, \beta)\bigr\}$, which simplifies directly to $\frac{1}{2}\min(1, \mu, \beta)$. This yields \eqref{eq:base_temporal_regularity} and concludes the proof.
\end{proof}

\section{Fully discrete approximation}\label{sec:fullydiscrete_approximation}
To numerically approximate \eqref{eq:SEE} via a standard Galerkin finite element method, let $\{V_h\}_{h \in (0,1]}$ be a quasi-uniform family of finite-dimensional subspaces of $\dot{H}^1$ consisting of continuous, piecewise linear functions defined over a regular triangulation of $\mathcal{O}$, with $h$ denoting the maximum mesh size. Following Kruse~\cite[Section~3]{MR3168284}, we define the generalized orthogonal projector $P_h \colon \dot{H}^{-1} \to V_h$ by 
\begin{equation}
    ( P_h x, y_h )_H = \langle x, y_h \rangle_{\dot{H}^{-1}, \dot{H}^1} \quad \text{for all } x \in \dot{H}^{-1}, y_h \in V_h,
\end{equation}
where $\langle \cdot, \cdot \rangle_{\dot{H}^{-1}, \dot{H}^1}$ denotes the duality pairing between $\dot{H}^{-1}$ and $\dot{H}^1$. Furthermore, let $A_h \colon V_h \to V_h$ be the discrete analogue of the operator $A$, defined by
\begin{equation}
    ( A_h x_h, y_h )_H = ( A^{1/2} x_h, A^{1/2} y_h )_H \quad \text{for all } x_h, y_h \in V_h.
\end{equation}
Since $A_h$ is self-adjoint and positive definite on the finite-dimensional space $V_h$, the operator $-A_h$ generates an analytic semigroup of contractions on $V_h$, denoted by $S_h(t) = e^{-t A_h}$. Projecting the continuous problem~\eqref{eq:SEE} onto the finite-dimensional subspace $V_h$ yields the following spatially semidiscrete initial value problem for the adapted process $X_h \colon [0,T]\times\Omega \to V_h$:
\begin{equation}\label{eq:SEE_semidiscrete}
\begin{aligned}
    dX_h(t) + A_h X_h(t)\,dt &= P_h F(t,X_h(t))\,dt + P_h G(t,X_h(t))\,dW(t), \quad t \in (0,T], \\
    X_h(0) &= P_h X_0.
\end{aligned}
\end{equation}

For the temporal discretization, we apply the linearly implicit Euler method to the semidiscrete problem~\eqref{eq:SEE_semidiscrete}. Let $N \in \mathbb{N}$ and introduce a uniform time step $k \coloneqq T/N$, defining the discrete time grid $t_n \coloneqq nk$ for $n=0,1,\dots,N$. The resulting fully discrete approximation $X_h^n \approx X(t_n)$ is generated by the recursion
\begin{equation}\label{eq:recursion}
    X_h^{n} = S_{h,k}X_h^{n-1} + k\,S_{h,k} P_h F(t_{n-1},X_h^{n-1}) + S_{h,k}P_h G(t_{n-1},X_h^{n-1})\,\Delta W^{n}, \quad n=1,\dots,N,
\end{equation}
with the initial condition $X_h^0 = P_h X_0$, where $S_{h,k} \coloneqq (I + k A_h)^{-1}$ and $\Delta W^n \coloneqq W(t_{n}) - W(t_{n-1})$. 

The discrete operator $S_{h,k}$ satisfies a standard smoothing property (see, e.g., Kruse~\cite[Eq.~(4.8)]{MR3168284}). For any $\rho \in [0,1]$, there exists a constant $C > 0$, independent of $h$, $k$, and $n$, such that 
\begin{equation}\label{eq:discrete_smoothing_estimate}
    \|A_h^{\rho} S_{h,k}^{n} x_h\|_H \le C t_{n}^{-\rho} \|x_h\|_H \quad \text{for all } x_h \in V_h \text{ and } n = 1, \dots, N.
\end{equation}
By an interpolation argument (cf.\ Kruse~\cite[Eqs.~(4.8) and (4.12)]{MR3168284}), this bound extends to negative norms evaluated at discrete time steps. For any $\eta \in [0, 1]$, there exists a constant $C>0$, independent of $h$, $k$, and $j$, such that
\begin{equation}\label{eq:negative_norm_smoothing_fully_discrete}
    \|S_{h,k}^j P_h x\|_H \le C t_j^{-\eta/2} \|x\|_{\dot{H}^{-\eta}} \quad \text{for all } x \in \dot{H}^{-\eta} \text{ and } j = 1, \dots, N.
\end{equation}
By iteratively expanding the recursion~\eqref{eq:recursion}, the fully discrete solution admits a discrete variation of constants formulation. Specifically, $X_h^n$ can be expressed explicitly as
\begin{equation}\label{eq:mild_fully_discrete}
    X_h^n = S_{h,k}^n\,X_h^0 
    + k \sum_{i=0}^{n-1} S_{h,k}^{\,n-i}\,P_h\,F(t_i,X_h^i) 
    + \sum_{i=0}^{n-1} S_{h,k}^{\,n-i}\,P_h\,G(t_i,X_h^i)\,\Delta W^{i+1}, \quad n=1,\dots,N.
\end{equation}

We now state our primary strong error estimate result for the fully discrete scheme.

\begin{theorem}[Strong error estimate for the fully discrete scheme]\label{thm:strong_convergence_of_fully_discrete_scheme}
Suppose that Assumptions~\ref{ass:initial_data}--\ref{ass:drift} hold. Let $p \in [2,\infty)$, $\beta \in (0,2)$, and $\mu \in (0,2]$ be such that $\mu > \beta-1$. Define $\nu \coloneqq \min\{\beta,\mu\}$. Then, for any arbitrarily small $\varepsilon \in (0, \beta)$, there exists a constant $C>0$, independent of the discretization parameters $h, k$ and the discrete time index $n \in \{1,\dots,N\}$, such that for all $n=1,\dots,N$,
\begin{equation}\label{eq:main-convergence}
    \|X(t_n) - X_h^n\|_{L^p(\Omega;H)} \le C \left(h^{\beta-\varepsilon} + k^{\frac{1}{2}\min\{\beta-\varepsilon, 1\}}\right) t_{n}^{-\frac{\beta-\nu}{2}} \left(1 + \|X_{0}\|_{L^{p}(\Omega;\dot H^{\nu})}\right),
\end{equation}
where $X(t_n)$ is the mild solution of \eqref{eq:SEE} evaluated at $t_n$, and $X_h^n$ is the fully discrete approximation defined by \eqref{eq:mild_fully_discrete}.
\end{theorem}

To maintain the flow of the exposition, we postpone the proof to Section~\ref{sec:proof_of_fully_discrete_result} after establishing the necessary preliminary results.

\section{Regularity results}\label{sec:regularity_results}

In this section, we establish the spatial and temporal regularity estimates for the mild solution. The following two lemmas will be useful for both the regularity results and the subsequent error estimates. The first lemma provides a Burkholder--Davis--Gundy (BDG) type inequality for $H$-valued stochastic integrals, which will be used repeatedly to estimate the moments of stochastic convolutions.

\begin{lemma}[{\cite[Proposition~2.12]{MR3154916}}]\label{lem:BDG}
Let \( p \ge 2 \) and \( 0 \le t_1 < t_2 \le T \).
Let \( \Phi \colon [0,T]\times\Omega \to {\mathrm{HS}(U_0, H)} \) be a predictable process such that
\begin{equation*}
    \mathbb{E}\left[ \left( \int_{t_1}^{t_2} \|\Phi(\sigma)\|_{\mathrm{HS}(U_0, H)}^2\, d\sigma \right)^{p/2} \right] < \infty .
\end{equation*}
Then the stochastic integral \( \int_{t_1}^{t_2} \Phi(\sigma)\, dW(\sigma) \) is well defined and satisfies
\begin{equation*}
    \mathbb{E}\left[ \left\| \int_{t_1}^{t_2} \Phi(\sigma)\, dW(\sigma) \right\|_H^p \right] \le C_p\, \mathbb{E}\left[ \left( \int_{t_1}^{t_2} \|\Phi(\sigma)\|_{\mathrm{HS}(U_0, H)}^2\, d\sigma \right)^{p/2} \right],
\end{equation*}
where the constant $C_p$ is given by
\begin{equation*}
    C_p = \left(\frac{p(p-1)}{2}\right)^{p/2} \left(\frac{p}{p-1}\right)^{p\left(\frac{p}{2}-1\right)} .
\end{equation*}
\end{lemma}

The following lemma collects several deterministic smoothing properties of the analytic $\mathcal{C}_0$-semigroup.

\begin{lemma}\label{lem:semigroup_estimates}
Let $\gamma^+ \coloneqq \max\{0,\gamma\}$ denote the positive part of $\gamma$, and let the operator $A$ and the corresponding analytic semigroup $\{S(t)\}_{t\ge 0}$ be as defined in Section~\ref{sec:introduction}. Then the following estimates hold:
\begin{enumerate}[label={\upshape(\roman*)}]
    \item For any $\gamma \in \mathbb{R}$, there exists a constant $C > 0$, depending on $\gamma$, such that
    \[
    \|A^{\gamma}S(t)\|_{\mathcal{L}(H)} \le C\, t^{-\gamma^+}, \quad t > 0.
    \]
    
    \item For any $\rho \in [0,1]$, there exists a constant $C > 0$, depending on $\rho$, such that
    \[
    \|A^{-\rho}(I-S(t))\|_{\mathcal{L}(H)} \le C\, t^{\rho}, \quad t > 0.
    \]
    
    \item For any $\rho \in [0,1]$, there exists a constant $C > 0$, depending on $\rho$, such that
    \[
    \int_{s}^{t} \|A^{\frac{\rho}{2}} S(t-\sigma)v\|_H^2 \,d\sigma \leq C \,(t - s)^{1-\rho} \|v\|_H^2 \quad \text{for all } v \in H, \; 0 \leq s < t.
    \]
    
    \item For any $\rho \in [0,1]$, there exists a constant $C > 0$, depending on $\rho$, such that
    \[
    \left\|A^\rho \int_{s}^{t} S(t-\sigma)v \, d\sigma \right\|_H \leq C \, (t - s)^{1-\rho} \|v\|_H \quad \text{for all } v \in H, \; 0 \leq s < t.
    \]   
\end{enumerate}
\end{lemma}
\begin{proof}
    The proof of (i) can be found in \cite[Lemma~2.1(i)]{NaikTripathi2026}, while the proofs for (ii)--(iv) are provided in \cite[Lemma~2.5]{MR3168284}.
\end{proof}

We now present the weighted spatial and temporal regularity estimates for the mild solution for $t>0$.
\begin{theorem}\label{thm:singular_regularity}
Suppose Assumptions~\ref{ass:initial_data}--\ref{ass:drift} hold and $p \ge 2$. Let $\beta \in (0,2)$ and $\mu \in (0,2]$ be such that $\mu > \beta-1$, and define $\nu \coloneqq \min\{\beta, \mu\}$. Then the mild solution $X(t)$ satisfies the following regularity properties:

\begin{enumerate}[label={\upshape(\roman*)}]
    \item 
    There exists a constant $C > 0$ such that, for all $t > 0$, 
    \begin{equation}\label{eq:weighted_spatial_regularity}
        \|X(t)\|_{L^{p}(\Omega;\dot{H}^{\beta})} \leq C \, t^{-\frac{\beta - \nu}{2}} \bigl( 1 + \|X_{0}\|_{L^{p}(\Omega;\dot H^{\nu})} \bigr).
    \end{equation}

    \item 
    There exists a constant $C > 0$ such that for $0 < s \le t \leq T$, 
    \begin{equation}\label{eq:weighted_temporal_regularity_combined}
        \|X(t) - X(s)\|_{L^p(\Omega;\dot{H}^{\max\{\beta-1, 0\}})} \leq C(t-s)^{\frac{\min\{\beta, 1\}}{2}} \, s^{-\frac{\beta -\nu}{2}} \bigl(1 + \|X_{0}\|_{L^{p}(\Omega;\dot H^{\nu})}\bigr).
    \end{equation}
\end{enumerate}
\end{theorem}

\begin{proof} 
(i) Taking the $L^{p}(\Omega;\dot H^{\beta})$-norm of the mild solution~\eqref{eq:mild_solution_SEE} yields
\begin{align}
    \|X(t)\|_{L^p(\Omega;\dot{H}^\beta)} 
    &\leq \|S(t)X_0\|_{L^p(\Omega;\dot{H}^\beta)} + \left\| \int_0^t S(t-\sigma) F(\sigma,X(\sigma)) \,d\sigma\right\|_{L^p(\Omega;\dot{H}^\beta)} \notag \\
    &\quad + \left\|\int_0^t S(t-\sigma) G(\sigma, X(\sigma))\, dW(\sigma)\right\|_{L^p(\Omega;\dot{H}^\beta)} \notag \\
    &\eqqcolon J_1 + J_2 + J_3. \label{eq:spatial_split}
\end{align}
For $J_1$, applying Lemma~\ref{lem:semigroup_estimates}(i) yields
\begin{equation}\label{eq:bound_J1}
    J_1 = \bigl\|A^{\frac{\beta-\nu}{2}}S(t) A^{\frac{\nu}{2}}X_0\bigr\|_{L^p(\Omega;H)} \leq C \, t^{-\frac{\beta-\nu}{2}} \|X_0\|_{L^p(\Omega;\dot{H}^\nu)}.
\end{equation}
Since the integrals diverge for $\beta \ge 1$ due to the singularity at $\sigma = t$, we isolate it by adding and subtracting the endpoint values $F(t,X(t))$ and $G(t,X(t))$. This yields $J_2 \le J_{2,a} + J_{2,b}$ and, via Lemma~\ref{lem:BDG}, $J_3 \le C(J_{3,a} + J_{3,b})$, where
\begin{align*}
    J_{2,a} &\coloneqq \int_0^t \bigl\| A^{\frac{\beta+1}{2}}S(t-\sigma) A^{-1/2} \bigl(F(\sigma, X(\sigma)) - F(t, X(t))\bigr) \bigr\|_{L^p(\Omega;H)} \, d\sigma, \\
    J_{3,a} &\coloneqq \left\| \left( \int_0^t \bigl\|A^{\frac{\beta}{2}}S(t-\sigma)\bigl(G(\sigma, X(\sigma)) - G(t, X(t))\bigr)\bigr\|_{\mathrm{HS}(U_0, H)}^2 \, d\sigma \right)^{\!1/2} \right\|_{L^p(\Omega)}, \\
    J_{2,b} &\coloneqq \left\| \int_0^t A^{\frac{\beta}{2}} S(t-\sigma) F(t, X(t)) \, d\sigma \right\|_{L^p(\Omega;H)}, \\
    J_{3,b} &\coloneqq \left\| \left( \int_0^t \bigl\|A^{\frac{1}{2}}S(t-\sigma)A^{\frac{\beta-1}{2}}G(t, X(t))\bigr\|_{\mathrm{HS}(U_0, H)}^2 \, d\sigma \right)^{\!1/2} \right\|_{L^p(\Omega)}.
\end{align*}
To bound $J_{2,a}$ and $J_{3,a}$, let $\alpha \coloneqq \min\{0, \beta-1\}$. Applying Assumptions~\ref{ass:diffusion}--\ref{ass:drift}, Lemma~\ref{lem:semigroup_estimates}(i), and \eqref{eq:base_temporal_regularity} for $\delta$ satisfying $\frac{1}{2}\max\{0, \beta-1\} < \delta < \frac{1}{2}\min\{1, \mu, \beta\}$, we obtain
\begin{align*}
    J_{2,a} &\leq C \int_0^t (t-\sigma)^{-\frac{\beta+1}{2}} \left( (t-\sigma)^{\beta/2} + (t-\sigma)^{\delta} \right) d\sigma \le C, \\
    J_{3,a} &\leq C \left( \int_0^t (t-\sigma)^{-(\beta-\alpha)} \left( (t-\sigma)^\beta + (t-\sigma)^{2\delta} \right) d\sigma \right)^{\!1/2} \le C.
\end{align*}
Evaluating the integrals in $J_{2,b}$ and $J_{3,b}$ using Lemma~\ref{lem:semigroup_estimates}(iv) and (iii), respectively, yields
\begin{align*}
    J_{2,b} &= \bigl\| A^{\frac{\beta-2}{2}} (I - S(t)) F(t, X(t)) \bigr\|_{L^p(\Omega;H)}, \\
    J_{3,b} &\le C \|G(t, X(t))\|_{L^p(\Omega; \mathrm{HS}(U_0, \dot{H}^{\beta-1}))}.
\end{align*}
We bound these terms in two steps.

\emph{Step 1: $\beta \in (0, 1]$.} 
Writing $A^{\frac{\beta-2}{2}} = A^{-\frac{1-\beta}{2}} A^{-1/2}$ and applying Lemma~\ref{lem:semigroup_estimates}(ii), Assumptions~\ref{ass:diffusion} and \ref{ass:drift}, and \eqref{eq:base_spatial_regularity}, we deduce:
\begin{align*}
    J_{2,b} &\leq C \, t^{\frac{1-\beta}{2}} \bigl(1 + \|X(t)\|_{L^p(\Omega;H)}\bigr) \leq C\bigl(1 + \|X_0\|_{L^p(\Omega; \dot{H}^\nu)}\bigr), \\
    J_{3,b} &\leq C\bigl(1 + \|X(t)\|_{L^p(\Omega; H)}\bigr) \leq C\bigl(1 + \|X_0\|_{L^p(\Omega; \dot{H}^\nu)}\bigr).
\end{align*}
Using $1 \le T^{\frac{\beta-\nu}{2}} t^{-\frac{\beta-\nu}{2}}$ and combining the bounds for $J_1, J_2,$ and $J_3$ establishes \eqref{eq:weighted_spatial_regularity} for $\beta \in (0, 1]$.

\emph{Step 2: $\beta \in (1, 2)$.} 
Here, applying Lemma~\ref{lem:semigroup_estimates}(ii) and Assumptions~\ref{ass:diffusion} and \ref{ass:drift} yields 
\[
    J_{2,b} + J_{3,b} \le C \bigl(1 + \|X(t)\|_{L^p(\Omega;\dot{H}^{\beta-1})}\bigr).
\]
We proceed via a bootstrap argument. Let $\tilde{\beta} \coloneqq \beta - 1 \in (0, 1)$ and $\tilde{\nu} \coloneqq \min\{\tilde{\beta}, \mu\} = \tilde{\beta}$ (since $\mu > \beta-1$). Applying the result of Step 1 to order $\tilde{\beta}$, the temporal singularity vanishes ($\tilde{\beta} - \tilde{\nu} = 0$), giving:
\begin{equation*}
    \|X(t)\|_{L^p(\Omega; \dot{H}^{\beta-1})} = \|X(t)\|_{L^p(\Omega; \dot{H}^{\tilde{\beta}})} \leq C\bigl(1 + \|X_0\|_{L^p(\Omega; \dot{H}^{\beta-1})}\bigr).
\end{equation*}
Since $\nu = \min\{\beta, \mu\} > \beta - 1$, the continuous embedding $\dot{H}^\nu \hookrightarrow \dot{H}^{\beta-1}$ closes the estimate. Finally using $1 \le T^{\frac{\beta-\nu}{2}} t^{-\frac{\beta-\nu}{2}}$ completes the proof of \eqref{eq:weighted_spatial_regularity} for $\beta \in (1,2)$.

\medskip
\noindent(ii) Let $r \coloneqq \max\{\beta-1, 0\}$, which implies $\beta - r = \min\{\beta, 1\}$. For the temporal regularity, we decompose the increment for $0 \le s \le t \le T$ as:
\begin{equation}\label{eq:temporal_decomp_combined}
    X(t) - X(s) = (S(t-s) - I) X(s) + \int_{s}^{t} S(t-\sigma)F(\sigma,X(\sigma))\,d\sigma + \int_{s}^{t} S(t-\sigma)G(\sigma, X(\sigma))\,dW(\sigma).
\end{equation}
For the first term, applying Lemma~\ref{lem:semigroup_estimates}(ii) alongside \eqref{eq:weighted_spatial_regularity} yields
\begin{align}
    \|(S(t-s) - I) X(s)\|_{L^p(\Omega;\dot{H}^r)} 
    &= \bigl\|A^{-\frac{\beta-r}{2}}(S(t-s) - I) A^{\frac{\beta}{2}} X(s)\bigr\|_{L^p(\Omega;H)} \notag \\
    &\le C (t-s)^{\frac{\min\{\beta, 1\}}{2}} \|X(s)\|_{L^p(\Omega;\dot{H}^\beta)} \notag \\
    &\le C (t-s)^{\frac{\min\{\beta, 1\}}{2}} s^{-\frac{\beta-\nu}{2}} \bigl(1 + \|X_0\|_{L^p(\Omega;\dot{H}^\nu)}\bigr). \label{eq:temporal_bound_first_combined}
\end{align}
For the drift term, let $\alpha \coloneqq \max\{\beta-2, -1\}$. Notice that $\alpha = \max\{\beta-1, 0\} - 1 = r - 1$, which gives $r - \alpha = 1$. Thus, we split the operator as $A^{r/2} = A^{1/2} A^{\alpha/2}$. Applying Assumption~\ref{ass:drift} and Lemma~\ref{lem:semigroup_estimates}(i) alongside \eqref{eq:weighted_spatial_regularity} yields:
\begin{align}
    \left\| \int_{s}^{t} S(t-\sigma)F(\sigma,X(\sigma))\,d\sigma \right\|_{L^p(\Omega;\dot{H}^r)} 
    &\le \int_s^t \bigl\|A^{1/2}S(t-\sigma)\bigr\|_{\mathcal{L}(H)} \|F(\sigma,X(\sigma))\|_{L^p(\Omega;\dot{H}^{\alpha})} \, d\sigma \notag \\
    &\le C \sup_{\sigma \in [s,t]} \|F(\sigma, X(\sigma))\|_{L^p(\Omega; \dot{H}^{\alpha})} \int_s^t (t-\sigma)^{-1/2} \, d\sigma \notag \\
    &\le C (t-s)^{1/2} \bigl(1 + \sup_{\sigma \in [s,t]} \|X(\sigma)\|_{L^p(\Omega; \dot{H}^r)}\bigr) \notag \\
    &\le C(t-s)^{\frac{\min\{\beta, 1\}}{2}} s^{-\frac{\beta-\nu}{2}} \bigl(1 + \|X_0\|_{L^p(\Omega; \dot{H}^\nu)}\bigr), \label{eq:temporal_bound_drift_combined}
\end{align}
For the diffusion term, we apply Lemma~\ref{lem:BDG}, Minkowski's integral inequality, and Lemma~\ref{lem:semigroup_estimates}(i). Splitting the operator as $A^{r/2} = A^{\frac{r-\beta+1}{2}} A^{\frac{\beta-1}{2}}$ and defining $\rho' \coloneqq \frac{r-\beta+1}{2} = \frac{1-\min\{\beta, 1\}}{2} \ge 0$, we get:
\begin{align}
    &\left\| \int_s^t S(t-\sigma) G(\sigma, X(\sigma)) \, dW(\sigma) \right\|_{L^p(\Omega; \dot{H}^r)} \notag \\
    &\le C \left( \int_s^t \bigl\| A^{\rho'} S(t-\sigma) \bigr\|_{\mathcal{L}(H)}^2 \bigl\| G(\sigma, X(\sigma)) \bigr\|_{L^p(\Omega; \mathrm{HS}(U_0, \dot{H}^{\beta-1}))}^2 \, d\sigma \right)^{1/2} \notag \\
    &\le C \sup_{\sigma \in [s,t]} \bigl\| G(\sigma, X(\sigma)) \bigr\|_{L^p(\Omega; \mathrm{HS}(U_0, \dot{H}^{\beta-1}))} \left( \int_s^t (t-\sigma)^{-2\rho'} \, d\sigma \right)^{1/2} \notag \\
    &\le C (t-s)^{\frac{\min\{\beta, 1\}}{2}} \sup_{\sigma \in [s,t]} \bigl\| G(\sigma, X(\sigma)) \bigr\|_{L^p(\Omega; \mathrm{HS}(U_0, \dot{H}^{\beta-1}))}. \label{eq:temporal_bound_diffusion_combined}
\end{align}
It remains to bound the supremum uniformly. By Assumption~\ref{ass:diffusion}:
\begin{itemize}
    \item \emph{If $\beta \in (0,1]$}: We have $\| G(\sigma, X(\sigma)) \|_{\mathrm{HS}(U_0, \dot{H}^{\beta-1})} \le C(1 + \|X(\sigma)\|_H)$. Using \eqref{eq:base_spatial_regularity} and the continuous embedding $\dot{H}^\nu \hookrightarrow H$, this is bounded by $C(1 + \|X_0\|_{L^p(\Omega; \dot{H}^\nu)})$.
    
    \item \emph{If $\beta \in (1,2)$}: We have $\| G(\sigma, X(\sigma)) \|_{\mathrm{HS}(U_0, \dot{H}^{\beta-1})} \le C(1 + \|X(\sigma)\|_{\dot{H}^{\beta-1}})$. Applying the result of part (i) at order $\beta' =\beta-1$ yields a corresponding parameter $\nu' = \min\{\beta-1, \mu\}$. Since $\mu > \beta-1$ by assumption, we have $\nu' = \beta-1$. Then \eqref{eq:weighted_spatial_regularity} gives:
    \[
        \|X(\sigma)\|_{L^p(\Omega; \dot{H}^{\beta-1})} \le C \bigl(1 + \|X_0\|_{L^p(\Omega; \dot{H}^{\beta-1})}\bigr).
    \]
    Furthermore, because $\nu = \min\{\beta, \mu\} > \beta-1$, the continuous embedding $\dot{H}^\nu \hookrightarrow \dot{H}^{\beta-1}$ allows us to bound $\|X_0\|_{L^p(\Omega; \dot{H}^{\beta-1})} \le C\|X_0\|_{L^p(\Omega; \dot{H}^\nu)}$.
\end{itemize}
In both cases, taking the supremum over $\sigma \in [0,T]$ yields the uniform bound:
\begin{equation}\label{eq:G_uniform_bound}
    \sup_{\sigma \in [0,T]} \bigl\| G(\sigma, X(\sigma)) \bigr\|_{L^p(\Omega; \mathrm{HS}(U_0, \dot{H}^{\beta-1}))} \le C\bigl(1 + \|X_0\|_{L^p(\Omega; \dot{H}^\nu)}\bigr).
\end{equation}
Substituting \eqref{eq:G_uniform_bound} back into \eqref{eq:temporal_bound_diffusion_combined} and using $1 \le T^{\frac{\beta-\nu}{2}} s^{-\frac{\beta-\nu}{2}}$ we get:
\begin{align}
    &\left\| \int_s^t S(t-\sigma) G(\sigma, X(\sigma)) \, dW(\sigma) \right\|_{L^p(\Omega; \dot{H}^r)} \notag \\
    &\quad \le C(t-s)^{\frac{\min\{\beta, 1\}}{2}} s^{-\frac{\beta-\nu}{2}} \bigl(1 + \|X_0\|_{L^p(\Omega; \dot{H}^\nu)}\bigr). \label{eq:temporal_bound_diffusion_final_combined}
\end{align}
Substituting the bounds \eqref{eq:temporal_bound_first_combined}, \eqref{eq:temporal_bound_drift_combined}, and \eqref{eq:temporal_bound_diffusion_final_combined} back into the decomposition \eqref{eq:temporal_decomp_combined} yields \eqref{eq:weighted_temporal_regularity_combined}.
\end{proof}

\section{Proof of Theorem~\ref{thm:strong_convergence_of_fully_discrete_scheme}}\label{sec:proof_of_fully_discrete_result}
To establish the strong error estimate for the fully discrete scheme, we need the following auxiliary lemmas. The first lemma provides error estimates for the corresponding deterministic linear problem:
\begin{equation}\label{eq:deterministic}
    \frac{d}{dt} u(t) + Au(t) = 0, 
    \quad u(0) = v, \quad t\in(0,T].
\end{equation}
Since the operators $P_h, A_h$, and $S_{h,k}$ were defined in Section~\ref{sec:fullydiscrete_approximation}, we collect error estimates for the fully discrete approximation utilizing the linear implicit Euler method by comparing the exact solution $u(t) = S(t)v$ the fully discrete solution $U_h^n = S_{h,k}^n P_h v$.

\begin{lemma}\label{lem:fully_discrete_error_estimates}
The following estimates hold for the fully discrete approximation for all $t_n > 0$ and $h, k \in (0,1]$:
\begin{enumerate}[label={\upshape(\roman*)}]
    \item Let $0 \leq \rho \leq 2$ and $-\min\{1,2-\rho\} \leq \eta \leq \rho$. Then, there exists a constant $C>0$ such that
    \[
        \bigl\|\bigl(S(t_{n})-S_{h,k}^{n}P_{h}\bigr)v\bigr\|_H \le C\bigl(h^{\rho}+k^{\frac{\rho}{2}}\bigr)\,t_n^{-\frac{\rho - \eta}{2}}\,\|v\|_{\dot H^{\eta}} \quad \text{for all } v \in \dot{H}^{\eta}.
    \]
    \item For any $0 \le \rho \le 1$, there exists a constant $C>0$ such that 
    \[
        \left\|\sum_{j=1}^{n}\int_{t_{j-1}}^{t_{j}} \bigl(S_{h,k}^{j}P_{h}-S(\sigma)\bigr)v\,d\sigma\right\|_H \le C\bigl(h^{2-\rho}+k^{\frac{2-\rho}{2}}\bigr)\,\|v\|_{\dot H^{-\rho}} \quad \text{for all } v \in \dot{H}^{-\rho}.
    \]
    \item For any $0 \le \rho \le 2$, there exists a constant $C>0$ such that 
    \[
        \left(\sum_{j=1}^{n}\int_{t_{j-1}}^{t_{j}} \bigl\|\bigl(S_{h,k}^{j}P_{h}-S(\sigma)\bigr)v\bigr\|_H^{2}\,d\sigma\right)^{\!1/2} \le C\bigl(h^{\rho}+k^{\frac{\rho}{2}}\bigr)\,\|v\|_{\dot H^{\rho-1}} \quad \text{for all } v \in \dot{H}^{\rho-1}.
    \]
\end{enumerate}
\end{lemma}

\begin{proof}
The error bound in (i) is a direct consequence of the deterministic error estimate established by Andersson et al.~\cite[Lemma~5.1]{MR3475837}. The bounds in (ii) and (iii) are derived in Wang~\cite[Lemma~3.2]{MR3649432}.
\end{proof}

We also need the following generalized discrete fractional Gronwall lemma given by Andersson et al.~\cite[Lemma~2.1]{MR3475837}, which is based on the foundational inequality proved by Elliott and Larsson~\cite[Lemma~7.1]{MR1122067}.

\begin{lemma}[{\cite[Lemma~2.1]{MR3475837}}]\label{lem:discrete_gronwall}
Let $T > 0$, $N \in \mathbb{N}$, $k = T/N$, and $t_n = n k$ for $0 \le n \le N$. If $(\varphi_n)_{n=0}^N$ is a sequence of non-negative real numbers satisfying
\[
    \varphi_n \le C_1 \left(1 + t_n^{-1+\theta}\right) + C_2 k \sum_{j=0}^{n-1} t_{n-j}^{-1+\alpha} \varphi_j, \quad 1 \le n \le N,
\]
for some constants $C_1, C_2 \ge 0$ and $\theta, \alpha > 0$, then there exists a constant $C = C(\theta, \alpha, C_2, T) > 0$ such that
\[
    \varphi_n \le C C_1 \left(1 + t_n^{-1+\theta}\right), \quad 1 \le n \le N.
\]
\end{lemma}

\begin{lemma}\label{lem:operator_splitting}
Let $0 \leq \rho \leq 2$ and $-\min\{1, 2-\rho\} \le \eta \le \rho$. Then, for any $\sigma \in [t_i, t_{i+1})$, there exists a constant $C > 0$ such that
\begin{equation}\label{eq:operator_splitting_bound}
    \bigl\| \bigl(S(t_n-\sigma) - S_{h,k}^{n-i}P_h\bigr) A^{-\frac{\eta}{2}} \bigr\|_{\mathcal{L}(H)} \le C \bigl( h^\rho + k^{\rho/2} \bigr) (t_n-\sigma)^{-\frac{\rho - \eta}{2}}.
\end{equation}
\end{lemma}

\begin{proof}
For $\sigma \in [t_i, t_{i+1})$, we add and subtract $S(t_{n-i})$. Applying the triangle inequality, the semigroup property, and the bounds from Lemmas~\ref{lem:semigroup_estimates} and \ref{lem:fully_discrete_error_estimates}, we obtain:
\begin{align}
    &\bigl\| \bigl(S(t_n-\sigma) - S_{h,k}^{n-i}P_h\bigr) A^{-\frac{\eta}{2}} \bigr\|_{\mathcal{L}(H)} \notag \\
    &\quad \le \bigl\| S(t_n-\sigma) \bigl(I - S(\sigma - t_i)\bigr) A^{-\frac{\eta}{2}} \bigr\|_{\mathcal{L}(H)} 
    + \bigl\| \bigl(S(t_{n-i}) - S_{h,k}^{n-i}P_h\bigr) A^{-\frac{\eta}{2}} \bigr\|_{\mathcal{L}(H)} \notag \\
    &\quad = \bigl\| S(t_n-\sigma) A^{\frac{\rho - \eta}{2}} \bigl(I - S(\sigma - t_i)\bigr) A^{-\frac{\rho}{2}} \bigr\|_{\mathcal{L}(H)} 
    + \bigl\| \bigl(S(t_{n-i}) - S_{h,k}^{n-i}P_h\bigr) A^{-\frac{\eta}{2}} \bigr\|_{\mathcal{L}(H)} \notag \\
    &\quad\le C (t_n-\sigma)^{-\frac{\rho - \eta}{2}} (\sigma - t_i)^{\frac{\rho}{2}} + C \bigl( h^\rho + k^{\frac{\rho}{2}} \bigr) t_{n-i}^{-\frac{\rho - \eta}{2}} \notag \\
    &\quad\le C \bigl( h^\rho + k^{\frac{\rho}{2}} \bigr) (t_n-\sigma)^{-\frac{\rho - \eta}{2}},\label{eq:operator_splitting_final}
\end{align}
where we have used the fact that $\sigma - t_i \le k$ and $t_{n-i} = t_n - t_i \ge t_n - \sigma$.
\end{proof}

\begin{proof}[\textbf{Proof of Theorem~\ref{thm:strong_convergence_of_fully_discrete_scheme}}]
From \eqref{eq:mild_solution_SEE} and \eqref{eq:mild_fully_discrete}, applying the triangle inequality yields
\begin{align}
\|X(t_{n}) - X_h^{n}\|_{L^{p}(\Omega;H)}
&\le \left\|\bigl(S(t_{n}) - S_{h,k}^{n}P_h\bigr) X_{0}\right\|_{L^{p}(\Omega;H)} \notag \\
&\quad+ \left\|\sum_{i=0}^{n-1}\int_{t_{i}}^{t_{i+1}}
\left( S(t_{n}-\sigma)F(\sigma,X(\sigma))
- S_{h,k}^{n-i}P_{h}F(t_i,X_h^i)\right) \, d\sigma \right\|_{L^{p}(\Omega;H)} \notag \\
&\quad+ \left\|\sum_{i=0}^{n-1}\int_{t_{i}}^{t_{i+1}}
\left( S(t_{n}-\sigma)G(\sigma, X(\sigma))
- S_{h,k}^{n-i}P_{h}G(t_i, X_h^i)\right) \, dW(\sigma) \right\|_{L^{p}(\Omega;H)} \notag \\
&\coloneqq J_{0} + J_{F} + J_{G}. \label{eq:bound_fullydiscrete}
\end{align}
To estimate $J_0$, we apply Lemma~\ref{lem:fully_discrete_error_estimates}(i), to obtain
\begin{align}
    J_0 &= \left\| \bigl(S(t_n) - S_{h,k}^n P_h\bigr) X_0 \right\|_{L^p(\Omega;H)} \notag \\
    &\le C \bigl( h^\beta + k^{\beta/2} \bigr) t_n^{-\frac{\beta-\nu}{2}} \|X_0\|_{L^p(\Omega; \dot{H}^\nu)}. \label{eq:bound_J0}
\end{align}
To bound $J_F$, we decompose the error into three parts:
\begin{align}
    J_F &\le \sum_{i=0}^{n-1} \int_{t_i}^{t_{i+1}} \left\| S_{h,k}^{n-i}P_h \bigl( F(\sigma, X(\sigma)) - F(t_i, X_h^i) \bigr) \right\|_{L^p(\Omega;H)} \, d\sigma \notag \\
    &\quad + \sum_{i=0}^{n-1} \int_{t_i}^{t_{i+1}} \left\| \bigl(S(t_n-\sigma) - S_{h,k}^{n-i}P_h\bigr) \bigl(F(\sigma, X(\sigma)) - F(t_n, X(t_n))\bigr) \right\|_{L^p(\Omega;H)} \, d\sigma \notag \\
    &\quad + \left\| \sum_{i=0}^{n-1} \int_{t_i}^{t_{i+1}} \bigl(S(t_n-\sigma) - S_{h,k}^{n-i}P_h\bigr) F(t_n, X(t_n)) \, d\sigma \right\|_{L^p(\Omega;H)} \notag \\
    &\coloneqq J_{F1} + J_{F2} + J_{F3}. \label{eq:bound_JF_new}
\end{align}
\noindent\textbf{Estimate of $J_{F1}$:} 
 To handle the temporal singularity at $t=0$, we split the sum into the initial term ($i=0$) and the remaining sum ($i \ge 1$), and apply \eqref{eq:negative_norm_smoothing_fully_discrete} with $\eta=1$. To estimate the initial term, we combine Assumption~\ref{ass:drift}, the triangle inequality, and \eqref{eq:base_spatial_regularity}, using the relation $k \cdot t_n^{-1/2} \le k^{1/2}$ to obtain:
\begin{align}
    &\int_0^k \left\| S_{h,k}^{n}P_h \bigl( F(\sigma, X(\sigma)) - F(0, X_h^0) \bigr) \right\|_{L^p(\Omega;H)} \, d\sigma \notag \\
    &\quad \le C \int_0^k t_n^{-1/2} \left( \sigma^{\beta/2} + 1 + \|X_0\|_{L^p(\Omega;H)} + \|X_0 - X_h^0\|_{L^p(\Omega;H)} \right) d\sigma \notag \\
    &\quad \le C k^{1/2} \left( k^{\beta/2} + 1 + \|X_0\|_{L^p(\Omega;\dot{H}^\nu)} + \|X_0 - X_h^0\|_{L^p(\Omega;H)} \right). \label{eq:bound_JF1_boundary}
\end{align}
For the remaining sum ($i \ge 1$), we invoke Assumption~\ref{ass:drift} and Theorem~\ref{thm:singular_regularity}(ii). Noting that $\sigma - t_i \le k$, we obtain:
\begin{align}
    &\sum_{i=1}^{n-1} \int_{t_i}^{t_{i+1}} \left\| S_{h,k}^{n-i}P_h \bigl( F(\sigma, X(\sigma)) - F(t_i, X_h^i) \bigr) \right\|_{L^p(\Omega;H)} \, d\sigma \notag \\
    &\quad \le C \sum_{i=1}^{n-1} \int_{t_i}^{t_{i+1}} t_{n-i}^{-1/2} \left( (\sigma-t_i)^{\beta/2} + \|X(\sigma) - X(t_i)\|_{L^p(\Omega;H)} + \|X(t_i) - X_h^i\|_{L^p(\Omega;H)} \right) d\sigma \notag \\
    &\quad \le C \sum_{i=1}^{n-1} k \, t_{n-i}^{-1/2} \Bigl( k^{\beta/2} + k^{\frac{\min\{\beta,1\}}{2}} t_i^{-\frac{\beta-\nu}{2}} \Bigr) \bigl(1 + \|X_0\|_{L^p(\Omega;\dot{H}^\nu)}\bigr) \notag \\
    &\qquad + C k \sum_{i=1}^{n-1} t_{n-i}^{-1/2} \|X(t_i) - X_h^i\|_{L^p(\Omega;H)}. \label{eq:bound_JF1_interior}
\end{align}
Combining \eqref{eq:bound_JF1_boundary} and \eqref{eq:bound_JF1_interior}, using the fact that $\sum_{i=1}^{n-1} k \, t_{n-i}^{-1/2} \, t_i^{-\frac{\beta-\nu}{2}} \le \int_0^{t_n} (t_n - \sigma)^{-1/2} \sigma^{-\frac{\beta-\nu}{2}} \, d\sigma \le C t_n^{\frac{1-\beta+\nu}{2}}$, we obtain
\begin{align}
    J_{F1} &\le C k \sum_{i=0}^{n-1} t_{n-i}^{-1/2} \|X(t_i) - X_h^i\|_{L^p(\Omega;H)} \notag \\
    &\quad + C k^{\frac{\min\{\beta,1\}}{2}} t_n^{-\frac{\beta-\nu}{2}} \bigl(1 + \|X_0\|_{L^p(\Omega;\dot{H}^\nu)}\bigr). \label{eq:bound_JF1_new}
\end{align}

\noindent\textbf{Estimate of $J_{F2}$:} 
We apply Lemma~\ref{lem:operator_splitting} with $\rho = \beta - \varepsilon$ and split the analysis into two cases:

\noindent \textit{Case 1: $\beta \in (0, 1]$.} 
Because $\rho < 1$, combining Lemma~\ref{lem:operator_splitting} with $\eta = -1$, Assumption~\ref{ass:drift}, and Theorem~\ref{thm:singular_regularity}(ii), we obtain:
\begin{align}
    J_{F2} &\le C(h^{\beta-\varepsilon} + k^{\frac{\beta-\varepsilon}{2}}) \sum_{i=0}^{n-1} \int_{t_i}^{t_{i+1}} (t_n-\sigma)^{-\frac{1+\beta-\varepsilon}{2}} \bigl\|F(\sigma, X(\sigma)) - F(t_n, X(t_n))\bigr\|_{L^p(\Omega;\dot{H}^{-1})} \, d\sigma \notag \\
    &\le C(h^{\beta-\varepsilon} + k^{\frac{\beta-\varepsilon}{2}}) \bigl(1 + \|X_0\|_{L^p(\Omega;\dot{H}^\nu)}\bigr) \int_0^{t_n} (t_n-\sigma)^{-\frac{1+\beta-\varepsilon}{2}} \Bigl( (t_n-\sigma)^{\beta/2} + (t_n-\sigma)^{\beta/2} \sigma^{-\frac{\beta-\nu}{2}} \Bigr) d\sigma \notag \\
    &\le C(h^{\beta-\varepsilon} + k^{\frac{\beta-\varepsilon}{2}}) t_n^{-\frac{\beta-\nu}{2}} \bigl(1 + \|X_0\|_{L^p(\Omega;\dot{H}^\nu)}\bigr). \notag
\end{align}

\noindent \textit{Case 2: $\beta \in (1, 2)$.} 
By choosing $\varepsilon$ sufficiently small such that $\varepsilon < \beta - 1$, we have $\rho = \beta - \varepsilon > 1$. Combining Lemma~\ref{lem:operator_splitting} with $\eta = \beta - 2$, Assumption~\ref{ass:drift}, and Theorem~\ref{thm:singular_regularity}(ii) yields:
\begin{align}
    J_{F2} &\le C(h^{\beta-\varepsilon} + k^{\frac{\beta-\varepsilon}{2}}) \int_0^{t_n} (t_n-\sigma)^{-\frac{2-\varepsilon}{2}} \bigl\|F(\sigma, X(\sigma)) - F(t_n, X(t_n))\bigr\|_{L^p(\Omega;\dot{H}^{\beta-2})} \, d\sigma \notag \\
    &\le C(h^{\beta-\varepsilon} + k^{\frac{\beta-\varepsilon}{2}}) \int_0^{t_n} (t_n-\sigma)^{-1+\frac{\varepsilon}{2}} \left( (t_n-\sigma)^{\beta/2} + \|X(\sigma) - X(t_n)\|_{L^p(\Omega;\dot{H}^{\beta-1})} \right) d\sigma \notag \\
    &\le C(h^{\beta-\varepsilon} + k^{\frac{\beta-\varepsilon}{2}}) \bigl(1 + \|X_0\|_{L^p(\Omega;\dot{H}^\nu)}\bigr) \int_0^{t_n} (t_n-\sigma)^{-1+\frac{\varepsilon}{2}} \Bigl( (t_n-\sigma)^{\beta/2} + (t_n-\sigma)^{1/2} \sigma^{-\frac{\beta-\nu}{2}} \Bigr) d\sigma \notag \\
    &\le C(h^{\beta-\varepsilon} + k^{\frac{\beta-\varepsilon}{2}}) t_n^{-\frac{\beta-\nu}{2}} \bigl(1 + \|X_0\|_{L^p(\Omega;\dot{H}^\nu)}\bigr), \notag
\end{align}
In both cases, we conclude:
\begin{equation}
    J_{F2} \le C(h^{\beta-\varepsilon} + k^{\frac{1}{2}\min\{\beta-\varepsilon, 1\}}) t_n^{-\frac{\beta-\nu}{2}} \bigl(1 + \|X_0\|_{L^p(\Omega;\dot{H}^\nu)}\bigr). \label{eq:bound_JF2_new}
\end{equation}

\noindent\textbf{Estimate of $J_{F3}$:} 
Using the substitutions $\tau = t_n - \sigma$ and $j = n-i$, we obtain:
\begin{equation*}
    J_{F3} = \left\| \sum_{j=1}^{n} \int_{t_{j-1}}^{t_j} \bigl(S(\tau) - S_{h,k}^{j}P_h\bigr) F(t_n, X(t_n)) \, d\tau \right\|_{L^p(\Omega;H)}.
\end{equation*}
Recall $r \coloneqq \max\{\beta-1, 0\}$, and let $\alpha \coloneqq \max\{\beta-2, -1\}$. Notice that $\alpha = \max\{\beta-1, 0\} - 1 = r - 1$. Applying Lemma~\ref{lem:fully_discrete_error_estimates}(ii) with $\rho = -\alpha$, and Assumption~\ref{ass:drift}, yields:
\begin{align}
    J_{F3} &\le C(h^{2+\alpha} + k^{\frac{2+\alpha}{2}}) \|F(t_n, X(t_n))\|_{L^p(\Omega;\dot{H}^\alpha)} \notag \\
    &\le C(h^{2+\alpha} + k^{\frac{2+\alpha}{2}}) \bigl(1 + \|X(t_n)\|_{L^p(\Omega;\dot{H}^r)}\bigr). \notag
\end{align}
We observe that $2+\alpha = 2 + \max\{\beta-2, -1\} = \max\{\beta, 1\}$. Because $h, k \le 1$ and $\max\{\beta, 1\} \ge \beta$, we have $h^{2+\alpha} \le h^\beta$ and $k^{\frac{2+\alpha}{2}} \le k^{\beta/2}$. Finally, utilizing the continuous embedding $\dot{H}^\beta \hookrightarrow \dot{H}^r$ alongside the spatial regularity estimate \eqref{eq:weighted_spatial_regularity}, we obtain:
\begin{equation}
    J_{F3} \le C(h^\beta + k^{\beta/2}) t_n^{-\frac{\beta-\nu}{2}} \bigl(1 + \|X_0\|_{L^p(\Omega;\dot{H}^\nu)}\bigr). \label{eq:bound_JF3_new}
\end{equation}
Substituting \eqref{eq:bound_JF1_new}, \eqref{eq:bound_JF2_new}, and \eqref{eq:bound_JF3_new} into \eqref{eq:bound_JF_new} and simplifying the dominant spatial and temporal rates yields:
\begin{align}
    J_F &\le C k \sum_{i=0}^{n-1} t_{n-i}^{-1/2} \|X(t_i) - X_h^i\|_{L^p(\Omega;H)} \notag \\
    &\quad + C \bigl(h^{\beta-\varepsilon} + k^{\frac{1}{2}\min\{\beta-\varepsilon, 1\}}\bigr) t_n^{-\frac{\beta-\nu}{2}} \bigl(1 + \|X_0\|_{L^p(\Omega;\dot{H}^\nu)}\bigr). \label{eq:bound_JF_final}
\end{align}
To bound the error $J_G$, we decompose it into three components:
\begin{align}
    J_G &\le \left\| \sum_{i=0}^{n-1} \int_{t_i}^{t_{i+1}} S_{h,k}^{n-i}P_h \bigl( G(t_i, X(t_i)) - G(t_i, X_h^i) \bigr) \, dW(\sigma) \right\|_{L^p(\Omega;H)} \notag \\
    &\quad + \left\| \sum_{i=0}^{n-1} \int_{t_i}^{t_{i+1}} S_{h,k}^{n-i}P_h \bigl( G(\sigma, X(\sigma)) - G(t_i, X(t_i)) \bigr) \, dW(\sigma) \right\|_{L^p(\Omega;H)} \notag \\
    &\quad + \left\| \sum_{i=0}^{n-1} \int_{t_i}^{t_{i+1}} \bigl( S(t_n-\sigma) - S_{h,k}^{n-i}P_h \bigr) G(\sigma, X(\sigma)) \, dW(\sigma) \right\|_{L^p(\Omega;H)} \notag \\
    &\coloneqq J_{G1} + J_{G2} + J_{G3}. \label{eq:bound_JG}
\end{align}

\noindent \textbf{Estimate of $J_{G1}$:} 
We first apply Lemma~\ref{lem:BDG} (Burkholder--Davis--Gundy inequality). Then, letting $\{\psi_m\}_{m=1}^\infty$ be an orthonormal basis of $U_0$, we expand the Hilbert--Schmidt norm and apply \eqref{eq:negative_norm_smoothing_fully_discrete} with $\eta = (1-\beta)^+ = \max\{0, 1-\beta\} \in [0,1)$ to obtain:
\begin{align}
    J_{G1} &\le C \left\| \left( \sum_{i=0}^{n-1} \int_{t_i}^{t_{i+1}} \bigl\| S_{h,k}^{n-i}P_h \bigl( G(t_i, X(t_i)) - G(t_i, X_h^i) \bigr) \bigr\|_{\mathrm{HS}(U_0, H)}^2 \, d\sigma \right)^{1/2} \right\|_{L^p(\Omega)} \notag \\
    &= C \left\| \left( \sum_{i=0}^{n-1} \int_{t_i}^{t_{i+1}} \sum_{m=1}^\infty \bigl\| S_{h,k}^{n-i}P_h \bigl( G(t_i, X(t_i)) - G(t_i, X_h^i) \bigr) \psi_m \bigr\|_H^2 \, d\sigma \right)^{1/2} \right\|_{L^p(\Omega)} \notag \\
    &\le C \left\| \left( \sum_{i=0}^{n-1} \int_{t_i}^{t_{i+1}} t_{n-i}^{-(1-\beta)^+} \sum_{m=1}^\infty \bigl\| \bigl( G(t_i, X(t_i)) - G(t_i, X_h^i) \bigr) \psi_m \bigr\|_{\dot{H}^{\min\{0, \beta-1\}}}^2 \, d\sigma \right)^{1/2} \right\|_{L^p(\Omega)} \notag \\
    &\le C \left\| \left( k \sum_{i=0}^{n-1} t_{n-i}^{-(1-\beta)^+} \bigl\| G(t_i, X(t_i)) - G(t_i, X_h^i) \bigr\|_{\mathrm{HS}(U_0, \dot{H}^{\min\{0, \beta-1\}})}^2 \right)^{1/2} \right\|_{L^p(\Omega)}. \notag
\end{align}
Applying Minkowski's integral inequality and Assumption~\ref{ass:diffusion}, we conclude:
\begin{align}
    J_{G1} &\le C \left( k \sum_{i=0}^{n-1} t_{n-i}^{-(1-\beta)^+} \bigl\| G(t_i, X(t_i)) - G(t_i, X_h^i) \bigr\|_{L^p(\Omega; \mathrm{HS}(U_0, \dot{H}^{\min\{0, \beta-1\}}))}^2 \right)^{1/2} \notag \\
    &\le C \left( k \sum_{i=0}^{n-1} t_{n-i}^{-(1-\beta)^+} \|X(t_i) - X_h^i\|_{L^p(\Omega; H)}^2 \right)^{1/2}. \label{eq:bound_JG1}
\end{align}

\noindent \textbf{Estimate of $J_{G2}$:} We apply Lemma~\ref{lem:BDG} (Burkholder--Davis--Gundy inequality), expand with the orthonormal basis $\{\psi_m\}_{m=1}^\infty$, and use \eqref{eq:negative_norm_smoothing_fully_discrete} with $\eta = (1-\beta)^+ \in [0, 1)$, just as we did for $J_{G1}$. By subsequently applying Minkowski's integral inequality, we obtain:
\begin{equation}
    J_{G2} \le C \left( \sum_{i=0}^{n-1} \int_{t_i}^{t_{i+1}} t_{n-i}^{-(1-\beta)^+} \bigl\| G(\sigma, X(\sigma)) - G(t_i, X(t_i)) \bigr\|_{L^p(\Omega; \mathrm{HS}(U_0, \dot{H}^{\min\{0, \beta-1\}}))}^2 \, d\sigma \right)^{1/2}. \label{eq:bound_JG2_intermediate}
\end{equation}
To handle the singularity at $t=0$, we split the sum into the initial term ($i=0$) and the remaining sum ($i \ge 1$). For the initial term ($i=0$), we apply Assumption~\ref{ass:diffusion} alongside \eqref{eq:base_spatial_regularity}. Noting that $t_n^{-(1-\beta)^+} \le k^{-(1-\beta)^+}$, and $1 - (1-\beta)^+ = \min\{\beta, 1\}$, we bound this component inside the square root as:
\begin{align}
    &\int_0^k t_n^{-(1-\beta)^+} \bigl\| G(\sigma, X(\sigma)) - G(0, X_0) \bigr\|_{L^p(\Omega; \mathrm{HS}(U_0, \dot{H}^{\min\{0, \beta-1\}}))}^2 \, d\sigma \notag \\
    &\quad \le \int_0^k k^{-(1-\beta)^+} C \bigl( \sigma^{\beta/2} + \|X(\sigma) - X_0\|_{L^p(\Omega; H)} \bigr)^2 \, d\sigma \notag \\
    &\quad \le C k^{\min\{\beta, 1\}} \bigl(1 + \|X_0\|_{L^p(\Omega; \dot{H}^\nu)}\bigr)^2. \label{eq:bound_JG2_boundary}
\end{align}
For the remaining sum ($i \ge 1$), we apply Assumption~\ref{ass:diffusion} and Theorem~\ref{thm:singular_regularity}(ii). Using the fact that $\sigma - t_i \le k$ and $k^\beta \le k^{\min\{\beta, 1\}}$ for $k \le 1$, we obtain
\begin{align}
    &\sum_{i=1}^{n-1} \int_{t_i}^{t_{i+1}} t_{n-i}^{-(1-\beta)^+} \bigl\| G(\sigma, X(\sigma)) - G(t_i, X(t_i)) \bigr\|_{L^p(\Omega; \mathrm{HS}(U_0, \dot{H}^{\min\{0, \beta-1\}}))}^2 \, d\sigma \notag \\
    &\quad \le C \sum_{i=1}^{n-1} k t_{n-i}^{-(1-\beta)^+} \left( k^\beta + k^{\min\{\beta, 1\}} t_i^{-(\beta-\nu)} \right) \bigl(1 + \|X_0\|_{L^p(\Omega; \dot{H}^\nu)}\bigr)^2 \notag \\
    &\quad \le C k^{\min\{\beta, 1\}} \bigl(1 + \|X_0\|_{L^p(\Omega; \dot{H}^\nu)}\bigr)^2 \int_0^{t_n} (t_n-\sigma)^{-(1-\beta)^+} \bigl( 1 + \sigma^{-(\beta-\nu)} \bigr) d\sigma \notag \\
    &\quad \le C k^{\min\{\beta, 1\}} \bigl(1 + \|X_0\|_{L^p(\Omega; \dot{H}^\nu)}\bigr)^2, \label{eq:bound_JG2_interior}
\end{align}
Combining \eqref{eq:bound_JG2_interior} with \eqref{eq:bound_JG2_boundary}, and substituting in \eqref{eq:bound_JG2_intermediate}, and using the fact $1 \le T^{\frac{\beta-\nu}{2}} t_n^{-\frac{\beta-\nu}{2}}$ yields 
\begin{equation}
    J_{G2} \le C k^{\frac{1}{2}\min\{\beta, 1\}} t_n^{-\frac{\beta-\nu}{2}} \bigl(1 + \|X_0\|_{L^p(\Omega; \dot{H}^\nu)}\bigr). \label{eq:bound_JG2}
\end{equation}

\noindent \textbf{Estimate of $J_{G3}$:} 
Let $\varepsilon \in (0, \min\{\beta, 1\})$. Applying Lemma~\ref{lem:BDG} (Burkholder--Davis--Gundy inequality), Minkowski's integral inequality, the standard Hilbert--Schmidt norm property $\|LM\|_{\mathrm{HS}(U_0,H)} \leq \|L\|_{\mathcal{L}(H)}\|M\|_{\mathrm{HS}(U_0,H)}$, and Lemma~\ref{lem:operator_splitting} with $\rho = \beta - \varepsilon$ and $\eta = \beta - 1$, alongside the uniform bound \eqref{eq:G_uniform_bound} established in Theorem~\ref{thm:singular_regularity}(ii), we estimate $J_{G3}$ as:
\begin{align}
    J_{G3} &\le C \left\| \left( \sum_{i=0}^{n-1} \int_{t_i}^{t_{i+1}} \bigl\| \bigl(S(t_n-\sigma) - S_{h,k}^{n-i}P_h\bigr) A^{\frac{1-\beta}{2}} \bigr\|_{\mathcal{L}(H)}^2 \bigl\| G(\sigma,X(\sigma)) \bigr\|_{\mathrm{HS}(U_0, \dot{H}^{\beta-1})}^2 \, d\sigma \right)^{1/2} \right\|_{L^p(\Omega)} \notag \\
    &\le C(h^{\beta-\varepsilon} + k^{\frac{\beta-\varepsilon}{2}}) \left( \int_0^{t_n} (t_n-\sigma)^{-(1-\varepsilon)} \bigl\| G(\sigma, X(\sigma)) \bigr\|_{L^p(\Omega; \mathrm{HS}(U_0, \dot{H}^{\beta-1}))}^2 \, d\sigma \right)^{1/2} \notag \\
    &\le C(h^{\beta-\varepsilon} + k^{\frac{\beta-\varepsilon}{2}}) \bigl(1 + \|X_0\|_{L^p(\Omega; \dot{H}^\nu)}\bigr) \left( \int_0^{t_n} (t_n-\sigma)^{-(1-\varepsilon)} \, d\sigma \right)^{1/2} \notag \\
    &\le C(h^{\beta-\varepsilon} + k^{\frac{1}{2}\min\{\beta-\varepsilon, 1\}}) t_n^{-\frac{\beta-\nu}{2}} \bigl(1 + \|X_0\|_{L^p(\Omega; \dot{H}^\nu)}\bigr),\label{eq:bound_JG3}
\end{align}
where we have used the fact that  $t_n^{\varepsilon/2} \le C t_n^{-\frac{\beta-\nu}{2}}$ (since $\varepsilon > 0$ and $\beta \ge \nu$) and $k^{\frac{\beta-\varepsilon}{2}} \le k^{\frac{1}{2}\min\{\beta-\varepsilon, 1\}}$ for $k \le 1$. Combining the bounds \eqref{eq:bound_JG1}, \eqref{eq:bound_JG2}, and \eqref{eq:bound_JG3}, and simplifying, we obtain 
\begin{equation}\label{eq:bound_JG_final}
    J_G \le C \left( k \sum_{i=0}^{n-1} t_{n-i}^{-(1-\beta)^+} \|X(t_i) - X_h^i\|_{L^p(\Omega; H)}^2 \right)^{1/2} 
    + C \bigl( h^{\beta-\varepsilon} + k^{\frac{1}{2}\min\{\beta-\varepsilon, 1\}} \bigr) t_n^{-\frac{\beta-\nu}{2}} \bigl(1 + \|X_0\|_{L^p(\Omega; \dot{H}^\nu)}\bigr).
\end{equation}
Combining the estimates \eqref{eq:bound_J0}, \eqref{eq:bound_JF_final}, and \eqref{eq:bound_JG_final} yields the bound for the total error $e_n \coloneqq \|X(t_n) - X_h^n\|_{L^p(\Omega;H)}$. Thus, we obtain:
\begin{equation*}
    e_n \le C \mathcal{E}_{h,k} t_n^{-\frac{\beta-\nu}{2}} + C k \sum_{i=0}^{n-1} t_{n-i}^{-1/2} e_i + C \left( k \sum_{i=0}^{n-1} t_{n-i}^{-(1-\beta)^+} e_i^2 \right)^{1/2},
\end{equation*}
where $\mathcal{E}_{h,k} \coloneqq \bigl( h^{\beta-\varepsilon} + k^{\frac{1}{2}\min\{\beta-\varepsilon, 1\}} \bigr) \bigl(1 + \|X_0\|_{L^p(\Omega; \dot{H}^\nu)}\bigr)$. Squaring both sides and applying the Cauchy--Schwarz inequality yields
\begin{align*}
    e_n^2 &\le C \mathcal{E}_{h,k}^2 t_n^{-(\beta-\nu)} + C \left( k \sum_{i=0}^{n-1} t_{n-i}^{-1/2} \right) \left( k \sum_{i=0}^{n-1} t_{n-i}^{-1/2} e_i^2 \right) + C k \sum_{i=0}^{n-1} t_{n-i}^{-(1-\beta)^+} e_i^2 \\
    &\le C \mathcal{E}_{h,k}^2 t_n^{-(\beta-\nu)} + C k \sum_{i=0}^{n-1} \bigl( t_{n-i}^{-1/2} + t_{n-i}^{-(1-\beta)^+} \bigr) e_i^2,
\end{align*}
where we used the fact that $k \sum_{i=0}^{n-1} t_{n-i}^{-1/2} \le \int_0^{t_n} (t_n-\sigma)^{-1/2} \, d\sigma = 2 t_n^{1/2} \le C$.

Let $\gamma \coloneqq \max\{1/2, (1-\beta)^+\}$. Since $\beta \in (0, 2)$, it follows that $\gamma \in [1/2, 1)$, meaning $t_{n-i}^{-1/2} + t_{n-i}^{-(1-\beta)^+} \le 2 t_{n-i}^{-\gamma}$.
Setting $\varphi_n \coloneqq e_n^2$ and applying the discrete Gr\"onwall inequality (Lemma~\ref{lem:discrete_gronwall}) with exponents $\theta = 1 -(\beta-\nu) > 0$ and $\alpha = 1 -\gamma > 0$ yields
\begin{equation*}
e_n^2 \le C \mathcal{E}_{h,k}^2 \bigl(1 + t_n^{-(\beta-\nu)}\bigr).
\end{equation*}
Finally, since $1 \le T^{\beta-\nu} t_n^{-(\beta-\nu)}$, taking the square root we obtain the final estimate:
\begin{equation*}
    \|X(t_n) - X_h^n\|_{L^p(\Omega;H)} \le C \bigl( h^{\beta-\varepsilon} + k^{\frac{1}{2}\min\{\beta-\varepsilon, 1\}} \bigr) t_n^{-\frac{\beta-\nu}{2}} \bigl(1 + \|X_0\|_{L^p(\Omega; \dot{H}^\nu)}\bigr),
\end{equation*}
concluding the proof.
\end{proof}

\section{Numerical experiments}\label{sec:numerical_results}

In this section, we present numerical experiments to verify our theoretical strong convergence rates. For clarity and consistency, we first define the general setting, the noise regularity, the discretization scheme, and the specific error metrics used to evaluate temporal and spatial convergence.

\paragraph{General setting.} Let $\mathcal{O} \coloneqq (0,1)$ and let $U = H = L^2(\mathcal{O})$. Define the linear operator $A: \mathcal{D}(A) \subset H \rightarrow H$ as the negative Laplacian equipped with homogeneous Dirichlet boundary conditions, given by $Au = -\frac{\partial^2 u}{\partial x^2}$ for $u \in \mathcal{D}(A) \coloneqq H^2(\mathcal{O}) \cap H_0^1(\mathcal{O})$. This setting provides an increasing sequence of real numbers $\eta_j = \pi^2j^2$ and an orthonormal basis $\{e_j = \sqrt{2}\sin(j\pi x), x \in \mathcal{O}\}_{j\in\mathbb{N}}$ such that $Ae_j = \eta_je_j$. Furthermore, let $W$ be a $Q$-Wiener process taking values in $H$. Assuming the covariance operator $Q$ shares the same eigenfunctions $\{e_j\}_{j\in\mathbb{N}}$ with $A$, the noise $W$ admits the Karhunen--Lo\`eve expansion
\begin{equation}\label{eq:KL_expansion}
    W(t) = \sum_{j=1}^{\infty} \sqrt{q_j} e_j \beta_j(t),
\end{equation}
where $q_j$ are the eigenvalues of $Q$, and $\{\beta_j(t)\}_{j\in\mathbb{N}}$ is a sequence of independent, real-valued standard Brownian motions.

\paragraph{Noise regularity.}
To control the spatial regularity of the noise, we define the covariance operator $Q$ as a fractional inverse power of the operator $A$. Specifically, we set $Q = A^{-r}$ for $r \ge 0$. Consequently, $Q$ and $A$ share the same eigenfunctions, and the eigenvalues of $Q$ are given by $q_j = \pi^{-2r} j^{-2r}$. We define the diffusion mapping $G$ as the Nemytskii operator associated with the scalar function $g \colon [0,T] \times \mathcal{O} \times \mathbb{R} \to \mathbb{R}$, satisfying 
\begin{align}
    \sup_{t \in [0,T], x \in \mathcal{O}} |g(t, x, 0)| &\le M, \\
    |g(t, x, y) - g(s, x, z)| &\le C \left( |t-s|^{\beta/2} + |y - z| \right),
\end{align}
for all $t, s \in [0,T]$, $x \in \mathcal{O}$, and $y, z \in \mathbb{R}$. As defined in \eqref{eq:Nemytskii_G} of Section~\ref{sec:introduction}, we have $(G(t, \phi)\psi)(x) \coloneqq g(t, x, \phi(x)) \psi(x)$ for all $\phi,\psi \in H$ and $x \in \mathcal{O}$.

\medskip
\noindent
\boxed{\text{Case 1: Space--time white noise }(\beta < 1/2).}
For space-time white noise, the covariance operator is the identity ($Q=I$, $r=0$), which implies $U_0 = H$. To verify Assumption~\ref{ass:diffusion}, we follow Case 1 of Lemma 4.2 established in \cite{NaikTripathi2026} with parameters $d=1$, $\theta = 0$, and $\mu = \nu = 0$. Equating the fractional index of the referenced lemma to $\beta - 1$, restricts the parameter to $\beta < 1/2$. Consequently, for all $\beta \in (0,1/2)$, Case 1 yields 
\begin{align*}
\|G(t,u) - G(s,v)\|_{\mathrm{HS}(H, \dot{H}^{\min(0,\beta-1)})} &\le C \left( |t-s|^{\beta/2} + \|u-v\|_H \right), \\
\|G(t,u)\|_{\mathrm{HS}(H, \dot{H}^{\beta-1})} &\le C \bigl(1 + \|u\|_H\bigr). 
\end{align*}
This satisfies the first and second conditions of Assumption~\ref{ass:diffusion} for the space-time white noise regime.

\medskip
\noindent
\boxed{\text{Case 2: Trace-class noise }(1 \le \beta < 2).}
For this regime, we assume the covariance operator $Q = A^{-r}$ for $r > 1/2$. This ensures the operator is of trace-class ($\sum_{j=1}^\infty q_j < \infty$). Because the basis functions are uniformly bounded ($\|e_j\|_{L^\infty(\mathcal{O};\mathbb{R})}^2 \le 2$), this implies $\sum_{j=1}^\infty q_j \|e_j\|_{L^\infty(\mathcal{O};\mathbb{R})}^2 < \infty$. Consequently, we apply Case 2 of \cite[Lemma~4.2]{NaikTripathi2026} with parameters $d=1$, $\theta = 0$, and $\mu = \nu = 0$. By setting the fractional index of the referenced lemma to $0$, Case 2 yields 
$$ \|G(t,u) - G(s,v)\|_{\mathrm{HS}(U_0, H)} \le C \left( |t-s|^{\beta/2} + \|u - v\|_H \right), $$
which satisfies the first condition of Assumption~\ref{ass:diffusion} since $\min(0, \beta-1) = 0$ for all $\beta \ge 1$. 

For $1 < \beta < 2$, the third condition requires bounding the linear growth in $\dot{H}^{\beta-1}$. We apply the commutative noise results from Jentzen and R\"ockner \cite[Subsection 4.3]{MR2852200} with $d=1$, $\rho = 2r$, $\nu = \pi^{-2r}$ and $\alpha = (\beta-1)/2$. Because the basis functions vanish at the boundary, Assumption~\ref{ass:diffusion} is fulfilled for 
\begin{equation}\label{eq:spectral_constraint}
    2r > 2\beta - 1, \quad \text{with } \beta \neq 1.5.
\end{equation}
To maintain a unified framework that guarantees this condition across all experiments, we define 
\begin{equation}\label{eq:mult_r_relation}
    2r = 2\beta - 1 + \delta, \quad \text{for } \delta > 0 \text{ (with } \delta = 0.001 \text{ in our implementation)}.
\end{equation}
We classify our numerical simulations into the following cases:
\begin{itemize}
    \item Setting $\beta=1.99$ yields $r = 1.4905$.  
    \item Setting $\beta=1.0$ yields $r = 0.5005$.
    \item Setting $\beta=0.4995$ yields $r=0$, i.e., $Q=I$ (space-time white noise).
\end{itemize}
Substituting the first two parameter pairs confirms that the required constraint \eqref{eq:spectral_constraint} is met for the trace-class noise. Furthermore, the third parameter pair satisfies the space-time white noise condition $\beta < 1/2$ established in Case 1. Hence, Assumption~\ref{ass:diffusion} is satisfied for all simulation cases $\beta \in \{1.99, 1.0, 0.4995\}$.

\paragraph{Discretization methodology.}
We spatially discretize the SPDEs given in this section using a standard Galerkin finite element method with continuous piecewise linear basis functions on a uniform grid of $N_x$ intervals (with mesh size $h = 1/N_x$), and in time using a linearly implicit Euler method with $N_t$ steps (with step size $k = T/N_t$). Furthermore, for the reference solution, the noise expansion \eqref{eq:KL_expansion} is truncated using a finite sum of $N_{x,\text{ref}} - 1$ terms, matching the dimension of the reference finite element space (cf.~Yan~\cite[Section~4]{MR2182132}). To isolate the spatial discretization error and measure the strong approximation error across the same Brownian paths, this identical truncated noise realization is applied across all coarsened spatial grids.

We measure the strong approximation error against a reference solution $U_{\text{ref}}$ computed on a highly refined grid with $N_{x,\text{ref}} = 256$ and $N_{t,\text{ref}} = 65536$. To approximate the expectations, we compute the average over $N_{\text{sim}} = 250$ independent sample paths. We isolate the discretization errors using a decoupled methodology:
\begin{itemize}
    \item The spatial grid is frozen at $N_{x} = 256$, with temporal refinements $N_t \in \{16, 64, 256, 1024, 4096\}$.
    \item The temporal grid is frozen at $N_{t} = 65536$, with spatial refinements $N_x \in \{4, 8, 16, 32, 64\}$.
\end{itemize}
All numerical simulations were implemented in MATLAB, utilizing parallel computing to efficiently process the Monte Carlo sample paths.

\paragraph{Error metrics and convergence rates.}
We measure the strong error in the $L^\infty(0,T; L^2(\Omega; H))$ norm. For the discrete time steps $t_n = n k$, the standard maximum strong error is approximated via a Monte Carlo average over $N_{\text{sim}}$ independent sample paths:
\begin{equation}\label{eq:standard_error}
    e_{\max}(h,k) \approx \max_{1 \le n \le N_t} \left( \frac{1}{N_{\text{sim}}} \sum_{j=1}^{N_{\text{sim}}} \bigl\| U^n(\omega_j) - U_{\text{ref}}(t_n, \omega_j) \bigr\|_H^2 \right)^{\!1/2}.
\end{equation}

For nonsmooth initial data, we compute the maximum strong error with a time weight to verify the theoretical convergence rates in Theorem~\ref{thm:strong_convergence_of_fully_discrete_scheme}, which contains an initial singular factor of $t_n^{-(\beta-\nu)/2}$. Because $\beta \in (0,2)$ and $\nu > 0$, this singularity exponent is bounded above by $1$. Rather than calculating the exact fractional regularity $\nu$ for each initial condition, we introduce a universal time weight $t_n$ and define the weighted error metric as:
\begin{equation}\label{eq:weighted_error}
    e_{w,\max}(h,k) \approx \max_{1 \le n \le N_t} \left( t_n \left( \frac{1}{N_{\text{sim}}} \sum_{j=1}^{N_{\text{sim}}} \bigl\| U^n(\omega_j) - U_{\text{ref}}(t_n, \omega_j) \bigr\|_H^2 \right)^{\!1/2} \right).
\end{equation}
Multiplying by $t_n$ uniformly absorbs the theoretical singularity at $t=0$. Since $t_n \in (0,T]$ and $1 - (\beta-\nu)/2 \ge 0$, the product $t_n^{1 - (\beta-\nu)/2}$ remains bounded. This neutralizes the transient initial blow-up near $t=0$, allowing us to recover the sharp convergence rates of $O(h^{\beta-\varepsilon} + k^{\frac{1}{2}\min\{\beta-\varepsilon, 1\}})$ for $\varepsilon>0$. 

Finally, to empirically estimate the experimental orders of convergence between successive grid refinement levels $j$ and $j+1$ (for $j = 1, \dots, 4$), we compute the standard spatial and temporal convergence rates ($R_x, R_t$) according to our decoupled methodology. By freezing the temporal parameter at a highly refined $k_{\text{ref}}$, the spatial rate is given by:
$$ R_x = \frac{\log(e_{\max}(h_j, k_{\text{ref}}) / e_{\max}(h_{j+1}, k_{\text{ref}}))}{\log(h_j / h_{j+1})}. $$
Conversely, freezing the spatial parameter at a highly refined $h_{\text{ref}}$ yields the temporal rate:
$$ R_t = \frac{\log(e_{\max}(h_{\text{ref}}, k_j) / e_{\max}(h_{\text{ref}}, k_{j+1}))}{\log(k_j / k_{j+1})}. $$
The time-weighted convergence rates $\tilde{R}_x$ and $\tilde{R}_t$ are evaluated using the identical logarithmic ratios applied to the weighted error metric $e_{w,\max}$.

\subsection{Smooth initial data}
\begin{example}\label{ex:highly_nonlinear_bounded}
As a specific instance of the general SPDE \eqref{eq:SPDE}, we consider the one-dimensional modified Langmuir model on $\mathcal{O} = (0,1)$ with $T=1$, driven by multiplicative noise and formulated as
\begin{equation}\label{eq:spde_langmuir_smooth}
\left\{
\begin{aligned}
    du(t,x) &= \left[ \frac{\partial^2 u}{\partial x^2}(t,x) + \frac{u(t,x)}{1 + |u(t,x)|} \right] dt + \frac{3(1 - u^2(t,x))}{1 + u^2(t,x)} \, dW(t), && (t,x) \in (0,1] \times (0,1), \\[1ex]
    u(t,x) &= 0, && (t,x) \in (0,1] \times \{0, 1\}, \\[1ex]
    u(0,x) &= \sin(\pi x), && x \in (0,1).
\end{aligned}
\right.
\end{equation}
\end{example}

Because the initial data $u_0(x) = \sin(\pi x)$ is smooth, the convergence rate is not constrained by the initial condition, but is entirely governed by the spatial regularity of the noise $\beta$. Figure~\ref{fig:loglog_langmuir_smooth} confirms the theoretical strong error bound of $O(h^{\beta-\varepsilon} + k^{\frac{1}{2}\min\{\beta-\varepsilon, 1\}})$ for $\varepsilon>0$, visually demonstrating the optimal spatial and temporal rates of convergence. Furthermore, since the smooth initial profile introduces no transient singularity, the standard metric (\textit{top row}) and the time-weighted metric (\textit{bottom row}) yield identical convergence behavior.

\begin{figure}[H]
    \centering
    \includegraphics[width=0.85\textwidth]{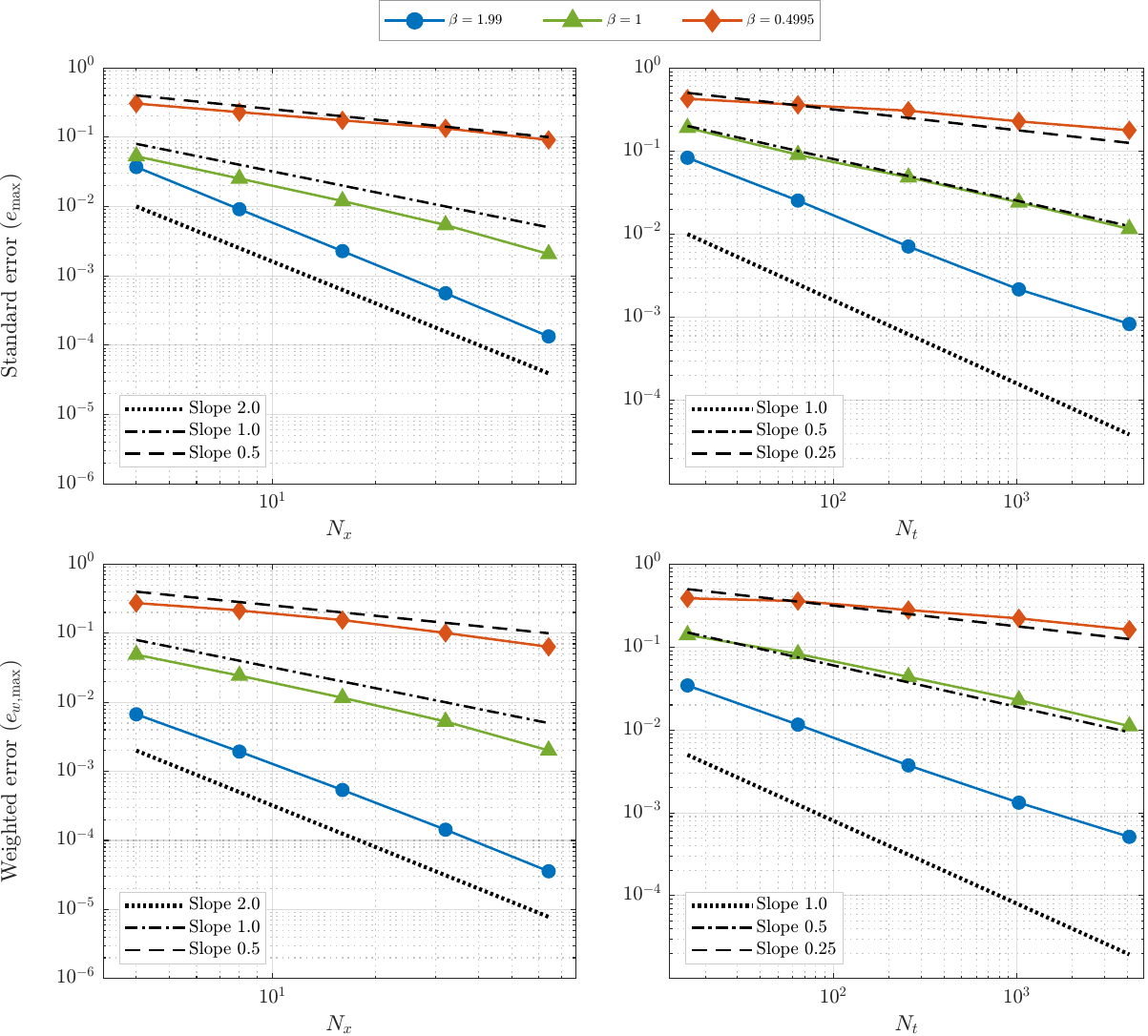}
    \caption{Spatial (\textit{left column}) and temporal (\textit{right column}) convergence rates across multiple parameter values $\beta$, evaluated using the standard strong error metric $e_{\max}$ (\textit{top row}) and the time-weighted error metric $e_{w,\max}$ (\textit{bottom row}).}
    \label{fig:loglog_langmuir_smooth}
\end{figure}

\subsection{Nonsmooth initial data}\label{subsec:nonsmooth_data_examples}
\begin{example}\label{ex:highly_nonlinear_nonsmooth}
For our next example, we consider the one-dimensional modified Langmuir model on $\mathcal{O} = (0,1)$ with $T=1$ and nonsmooth initial data, formulated as
\begin{equation}\label{eq:num_spde_multiplicative_rational}
\left\{
\begin{aligned}
    du(t,x) &= \left[ \frac{\partial^2 u}{\partial x^2}(t,x) + \frac{u(t,x)}{1 + |u(t,x)|} \right] dt + \frac{3(1 - u^2(t,x))}{1 + u^2(t,x)} \, dW(t), && (t,x) \in (0,1] \times (0,1), \\[1ex]
    u(t,x) &= 0, && (t,x) \in (0,1] \times \{0, 1\}, \\[1ex]
    u(0,x) &= \mathds{1}_{(0, 0.5)}(x), && x \in (0,1).
\end{aligned}
\right.
\end{equation}
\end{example}

Since the initial condition $u_0(x) = \mathds{1}_{(0, 0.5)}(x) \in \dot{H}^{1/2-\epsilon}(\mathcal{O})$ for $\epsilon > 0$. As illustrated in the top row of Figure~\ref{fig:loglog_langmuir_nonsmooth_step}, the lack of initial data regularity prevents the numerical scheme from achieving the optimal standard convergence rates. However, multiplying the error by the discrete time weight $t_n$ compensates for the initial singularity. The bottom row of Figure~\ref{fig:loglog_langmuir_nonsmooth_step} demonstrates that the time-weighted metric recovers the sharp theoretical convergence bound of $O(h^{\beta-\varepsilon} + k^{\frac{1}{2}\min\{\beta-\varepsilon, 1\}})$ for $\varepsilon>0$ across all tested noise regularities.

\begin{figure}[H]
    \centering
    \includegraphics[width=0.85\textwidth]{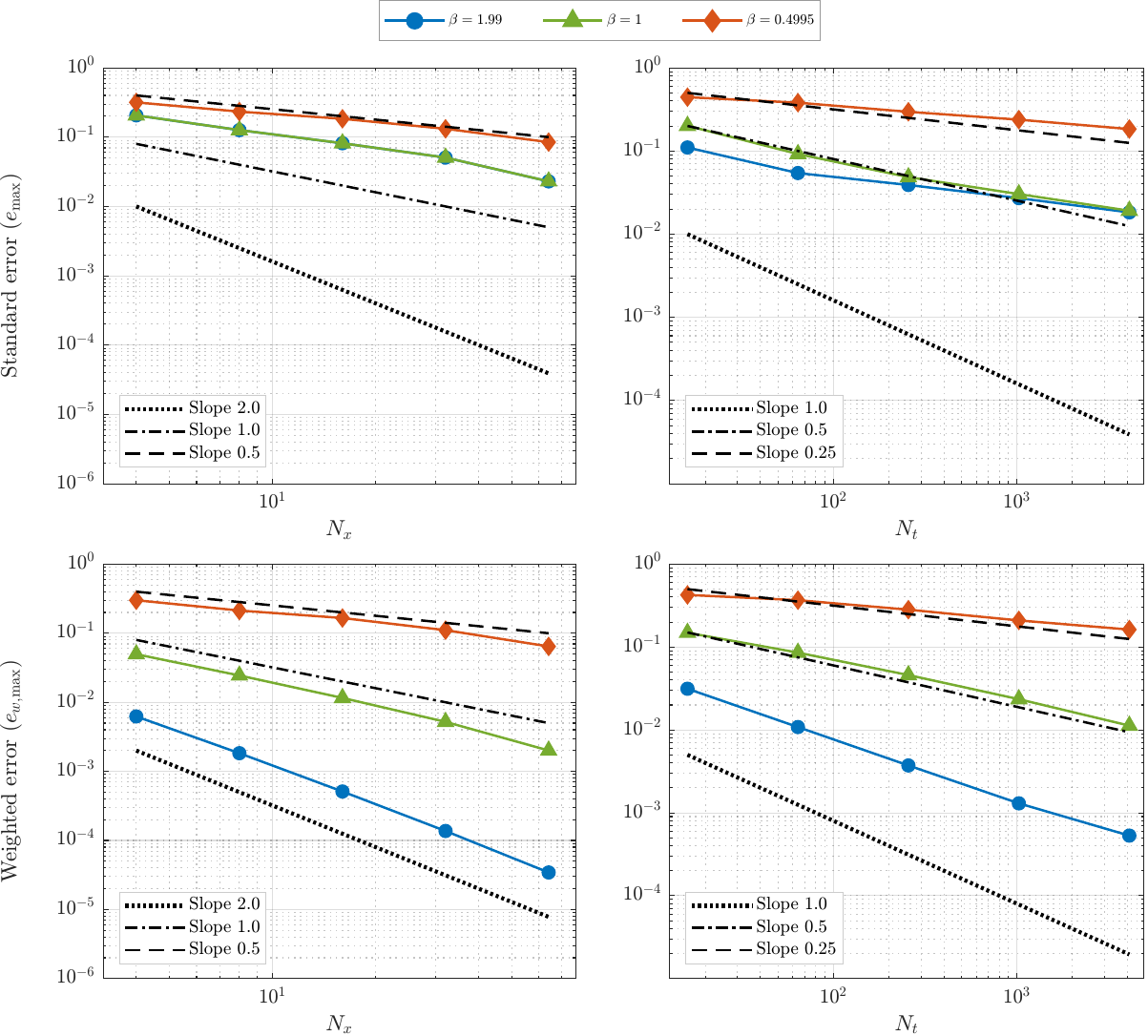}
    \caption{Spatial (\textit{left column}) and temporal (\textit{right column}) convergence rates across multiple parameter values $\beta$, evaluated using the standard strong error metric $e_{\max}$ (\textit{top row}) and the time-weighted error metric $e_{w,\max}$ (\textit{bottom row}).}
    \label{fig:loglog_langmuir_nonsmooth_step}
\end{figure}

\begin{example}\label{ex:pam_discontinuous}
As another instance of the general SPDE \eqref{eq:SPDE}, we consider the one-dimensional Parabolic Anderson Model (PAM) on $\mathcal{O} = (0,1)$ with $T=1$ and discontinuous initial data, formulated as
\begin{equation}\label{eq:num_spde_pam}
\left\{
\begin{aligned}
    du(t,x) &= \frac{\partial^2 u}{\partial x^2}(t,x) \, dt + u(t,x) \, dW(t), && (t,x) \in (0,1] \times (0,1), \\[1ex]
    u(t,x) &= 0, && (t,x) \in (0,1] \times \{0, 1\}, \\[1ex]
    u(0,x) &= \mathds{1}_{(0, 0.5)}(x), && x \in (0,1).
\end{aligned}
\right.
\end{equation}
\end{example}

For the Parabolic Anderson Model, we again utilize the discontinuous initial condition $u_0(x) = \mathds{1}_{(0, 0.5)}(x) \in \dot{H}^{1/2-\epsilon}(\mathcal{O})$ for $\epsilon > 0$. As illustrated in the top row of Figure~\ref{fig:loglog_pam_nonsmooth_step}, the lack of initial data regularity prevents the numerical scheme from achieving the standard convergence rates. However, multiplying the error by the discrete time weight $t_n$ compensates for the initial singularity. The bottom row of Figure~\ref{fig:loglog_pam_nonsmooth_step} demonstrates that the time-weighted metric recovers the sharp theoretical convergence rate of $O(h^{\beta-\varepsilon} + k^{\frac{1}{2}\min\{\beta-\varepsilon, 1\}})$ for $\varepsilon>0$ across all tested noise regularities.

\begin{figure}[H]
    \centering
    \includegraphics[width=0.85\textwidth]{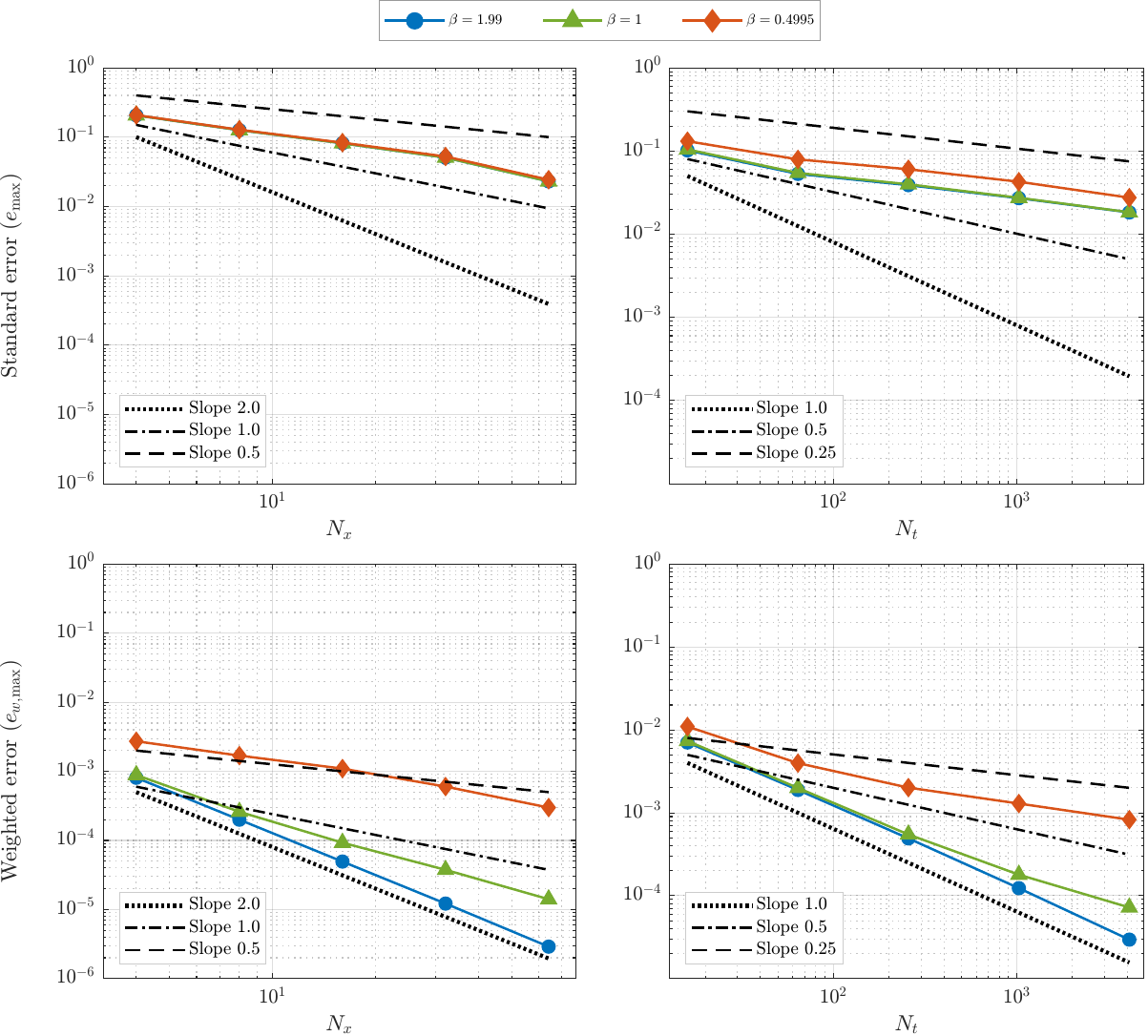}
    \caption{Spatial (\textit{left column}) and temporal (\textit{right column}) convergence rates across multiple parameter values $\beta$, evaluated using the standard strong error metric $e_{\max}$ (\textit{top row}) and the time-weighted error metric $e_{w,\max}$ (\textit{bottom row}).}
    \label{fig:loglog_pam_nonsmooth_step}
\end{figure}

\subsubsection*{Acknowledgments}
The authors acknowledge the support provided by the Indian Institute of Technology Goa, India.

\bibliographystyle{abbrv} 
\bibliography{references.bib}

\appendix

\end{document}